\documentclass[12pt]{amsart}
\usepackage{amsmath,amsfonts,amssymb} \usepackage{color}
\usepackage[all,cmtip]{xy}
\usepackage{graphicx}

\newcommand{\pd}{\partial}

 \newcommand{\PP}{\mathbb{P}}  \newcommand{\cP}{\mathcal{P}}
 \newcommand{\B}{\mathcal{B}}  
 \newcommand{\I}{\mathcal{I}} \newcommand{\cJ}{\mathcal{J}}
 \newcommand{\W}{\mathcal{W}} \newcommand{\cO}{\mathcal{O}}
 \newcommand{\Z}{\mathbb{Z}}
 \newcommand{\C}{\mathbb{C}}  \newcommand{\cC}{\mathcal{C}}
 \newcommand{\bone}{\mathbf{1}}
\newcommand{\Mbar}{\overline{\mathcal M}}

 \DeclareMathOperator{\coeff}{coeff}
 \DeclareMathOperator{\vdim}{vdim}  \DeclareMathOperator{\Spec}{Spec}
 \DeclareMathOperator{\diag}{diag}  \DeclareMathOperator{\Li}{Li}
  \DeclareMathOperator{\Ch}{Ch}

\newcommand{\be}{\begin{equation}}
\newcommand{\ee}{\end{equation}}
\newcommand{\bea}{\begin{eqnarray}}
\newcommand{\eea}{\end{eqnarray}}
\newcommand{\ben}{\begin{eqnarray*}}
\newcommand{\een}{\end{eqnarray*}}

\newcommand{\half}{\frac{1}{2}}

\newtheorem{cor}{Corollary}[section] \newtheorem{conjecture}{Conjecture}

 \newtheorem{lem}[cor]{Lemma}
 \newtheorem{prop}[cor]{Proposition}
 \newtheorem{thm}[cor]{Theorem}
\theoremstyle{remark}
 \newtheorem{defn}[cor]{Definition}
 \newtheorem{ex}[cor]{Example}
 \newtheorem{rmk}[cor]{Remark}
 \newtheorem{conjdefn}[cor]{Conjecture/Definition}

\definecolor{A}{rgb}{.75,1,.75}

 \makeindex
\begin{document}
\title[GV Invariants, Hilbert Schemes and Quasimodular Forms]
{Gopakumar-Vafa BPS Invariants, Hilbert Schemes and Quasimodular Forms. I.}
\author{Shuai Guo}
\address{Beijing International Center for
Mathematical Research\\Peiking University\\Beijng, 100871, China}
\email{gs0202@gmail.com}
\author{Jian Zhou}
\address{Department of Mathematical Sciences\\Tsinghua University\\Beijng, 100084, China}
\email{jzhou@math.tsinghua.edu.cn}

\begin{abstract}
We prove a closed formula for leading Gopakumar-Vafa BPS invariants of local Calabi-Yau geometries given by the canonical line bundles of toric Fano surfaces.
It shares some similar features with G\"ottsche-Yau-Zaslow formula:
Connection with Hilbert schemes, connection with quasimodular forms,
and quadratic property after suitable transformation.
In Part I of this paper we will present the case of projective plane,
more general cases will be presented in Part II.
\end{abstract}

\maketitle

 \section{Introduction}

The problem of counting curves in algebraic varieties  dates back
to the nineteenth century.
Classical examples include the famous 27 lines on a cubic surface
and the 2875 lines on a quintic 3-fold.
Through the interaction with string theory,
spectacular progresses in this classical branch of algebraic geometry
have been made over the years since 1990's.
First of all,
Gromov-Witten theory and its various variants have laid the foundation
for the modern treatment of many classical enumerative problems
and greatly expanded the range of enumerative problems being considered.
See \cite{PT} for an introduction to some of the exciting developments.

The study of the problems of counting curves in surfaces and Calabi-Yau 3-folds
have served both as motivations and applications of Gromov-Witten theory.
We will recall some results in these directions in \S \ref{sec:Motivations1} and \S \ref{sec:Motivations2}
to serve as motivations for this work.
We emphasize on the following three salient features for the curving counting problems
for algebraic surfaces:
connection with Hilbert schemes,
connection with quasimodular forms,
and quadratic properties of the node polynomials after
suitable transformation.
Our main results   will show that these features are also shared
in the curve counting problems of some noncompact Calabi-Yau 3-folds,
arising as the the total space of the canonical line bundles
of toric Fano surfaces.

From their constructions,
Gromov-Witten invariants are rational numbers in general.
It is amazing that for Calabi-Yau 3-folds,
physicists \cite{GV} have suggested a way
to convert them into integer invariants,
called the Gopakumar-Vafa BPS invariants.
These invariants in general are not the ``numbers" of algebraic curves of some fixed degree and genus,
embedded in the Calabi-Yau 3-folds,
but serve as a useful alternative.
Closed formulas for these invariants are very desirable to find,
our result adds an item to the very short list of examples known at present.

For compact Calabi-Yau 3-folds such as the quintic 3-fold,
mathematical computations of their Gromov-Witten invariants are available only in geneus 0 \cite{Giv, LLY} and genus 1 \cite{Zin},
based on the string theorists' prediction using mirror symmetry \cite{CDGP, BCOVK}.
Unfortunately it is not known at present how to generalize the computations
to higher genera (see Chang-Li \cite{CL} for some recent progresses.)

On the other hand,
in the noncompact setting
a method to compute the Gromov-Witten invariants of toric Calabi-Yau 3-folds
called the topological vertex \cite{AKMV} has been developed by string theorists,
based on duality with Chern-Simons theory \cite{Wit1, Wit2, Gop-Vaf, Oog, Mar-Vaf}.
A mathematical theory of the topological vertex has been developed in \cite{LLLZ}
to justify this method.
Based on this method the Gopakumar-Vafa integrality has been established
in this case in \cite{P, Ko1, Ko2}.
In this paper we will study the Gopakumar-Vafa invariants of
the Calabi-Yau 3-folds that are the total spaces of the canonical lines bundle of
toric Fano surfaces.
In this case the duality of their Gromov-Witten invariants with Chern-Simons link invariants
was more straightforward and was developed in \cite{AMV, I},
and its mathematical proof was presented in \cite{zhou3}.
As mentioned in \cite{AMV},
the complexity of computations increases very fast:
For degree $d=12$,
it involves evaluating $18239$ terms,
while for degree $20$ involves $943304$ terms.
It turns out that the colored HOMFLY polynomials play a key role in the computations.
They can be given by some specializations of the skew Schur functions,
and this leads us to the theory basic hypergeometric series
and some simple results in that theory can be applied.
This is one of the key useful technical tools in our work.

As mentioned above,
it is very interesting to compare our results with the corresponding results on algebraic surfaces \cite{Yau-Zas, Got}.
They share many common features,
and their similarities and differences may shed some lights on each other.
First of all,
both cases involve the  theory of quasimodular forms \cite{Kan-Zag},
which first appeared in the counting problems on elliptic curves \cite{Dij}.
Secondly,
both cases involve Hilbert schemes of points on surfaces.
Thirdly,
both involve polynomials that after suitable transformations become quadratic.

Let us present here our result for the special case of $X = \kappa_{\PP^2} = \cO_{\PP^2}(-3)$,
the total space of the canonical line bundle of $\PP^2$,
to give a sample of our results.
In an earlier work by the second author \cite{zhou1},
it was observed based on the table in \cite{AMV} that after a suitable transformation,
the Gopakumar-Vafa invariants of $\kappa_{\PP^2}$ become quadratic
polynomials in the degree $d$:
\bea
&& M^0_d = \frac{1}{2}(d^2+3d+2), \label{eqn:M0} \\
&& M^1_d = \frac{1}{2} (d^2+3d-4), \;\; (d \geq 3) \\
&& M_d^2 = \frac{3}{2} (d^2 +3d-6), \;\;\; (d \geq 4) \\
&& M_d^3 = 3(d^2+3d)-24, \;\;\; (d \geq 5) \\
&& M_d^4 = 6 (d^2+3d-11), \;\; (d \geq 6) \\
&& M_d^5 = \frac{21}{2}(d^2+3d)-144, \;\;\; (d \geq 7) \\
&& M_d^6 = 20(d^2+3d-16), \;\;\; (d \geq 8) \\
&& M_d^7 = \frac{67}{2}(d^2+3d)-626, \;\;\; (d \geq 9) \\
&& M_d^8 = \frac{117}{2}(d^2+3d)-1233, \;\;\; (d \geq 10) \label{eqn:M8}
\eea
It was conjectured in \cite{zhou1} that in general for $d \geq m+2$,
\ben
&& M^m_d = \frac{a(m)}{2}(d^2+3d)-b(m)
\een
for some positive integers $a(m)$ and $b(m)$.
(Similar observations were also made for $\kappa_{\PP^1 \times \PP^1}$.)
In this paper we will prove this conjecture,
and obtain the following closed formula for $a(m)$ and $b(m)$ in terms
of quasimodular forms (cf. Theorem \ref{thm:Mdelta}):
\begin{multline} \label{eqn:P2Case}
\frac{1}{(1-q)^2(1-q^2)}\sum_{m=0}^\infty M^m_d q^m \\
= \frac{1}{\prod_{n=1}^\infty (1-q^n)^3}
\biggl( \frac{(d+1)(d+2)}{2}
- 3 G_2(q) \biggr),
\end{multline}
where $G_2(q) = \sum_{n=1}^\infty \sum_{d|n} d q^n$ is the second Einsenstein series.
The right-hand side of the above formula has a striking similarity with
the Yau-Zaslow formula \cite{Yau-Zas}.
We will also show that \eqref{eqn:P2Case} verifies the predictions
made by Katz-Klemm-Vafa \cite{KKV} based on M-theory.
In this paper,
we will also show that our method can be used to obtain some results
not predicted by \cite{KKV}
and not observed in \cite{zhou1}.
In Theorem \ref{thm:Mdelta2} we prove the following formula holds for $0 \leq m \leq 2d-5$:
\be
\begin{split}
& M^m_d  = \coeff\biggl(\frac{[1]^2[2]}{[\infty]!^3} \biggl(\binom{d+2}{2}
- 3 \sum_{i\geq 1} \frac{q^{i}}{(1-   q^{i})^2}\biggr)\\
&- \frac{3\cdot q^{d-1}[2][3]}{[\infty]!^3}\biggl(\binom{d+1}{2}
- 3 \sum_{i\geq 1} \frac{q^{i}}{(1-   q^{i})^2}
-3\frac{q^{3}}{1-q^{3}} \biggr),q^m \biggr).
\end{split}
\ee
Based on such results we will also make some observations and speculations
about higher order corrections:
The quadratic property of the transformed GV invariants seems to persist
also in higher order corrections.

We now explain the title of this paper.
On the left-hand-side of \eqref{eqn:P2Case} we have the generating series
for transformed Gopakumar-Vafa invariants,
on the right-hand side,
\be
\frac{1}{\prod_{n=1}^\infty (1-q^n)^3} = \sum_{n=0}^\infty \chi((\PP^2)^{[n]}) q^n,
\ee
where $(\PP^2)^{[n]}$ is the Hilbert schemes of zero dimensional subschemes
of length $n$ in $\PP^2$.
Also on the right-hand, $G_2$ is a quasimodular form.

Furthermore,
the term $\frac{(d+1)(d+2)}{2}$ on the right-hand side has the combinatorial
meaning of number of lattice points in the triangle that describes
the toric geometry of $\PP^2$.
In Part II of this paper,
we will present the same type of formulas
for all toric Fano surfaces.
In this general formula,
the meaning of each term on the right-hand side will become
more transparent.
More precisely,
let $S$ be a toric Fano surface and $X = \kappa_S$.
For $\beta \in H_2(S;\Z)$,
there is a $\Z$-valued function $P^S(\beta)$ quadratic in $\beta$
such that the generating series of the leading transformed Gopakumar-Vafa
invariants is given by:
\be
\frac{[1]^2[2]}{[\infty]!^{e(S)}} \cdot
\biggl( P^S(\beta) - e(S) \cdot
\sum_{i\geq 1} \frac{q^{i}}{(1-   q^{i})^2}\biggr).
\ee
The combinatorial meaning of $P^S(\beta)$ is the number of lattice
points in the integral lattices associated to $S$,
with size controlled by $\beta$.

We work with Gromov-Witten invariants of noncompact Calabi-Yau 3-folds in this paper.
It is also very interesting to compare our results with the corresponding results
for compact Calabi-Yau 3-fold such as the quintic 3-fold.
We hope our results can have some counterparts in the compact case.
In a work in progress,
we extend our results to open string invariants.

The rest of this paper is arranged as follows.
In Section 2 we recall some work on counting curves in algebraic surfaces.
In Section 3 we recall the definition of Gopakumar-Vafa invariants
and the observations on the quadratic properties of
the transformed Gopakumar-Vafa invariants.
In Section 4 we recall some technical preliminaries.
In Section 5 we will prove  \eqref{eqn:P2Case}
and show that it matches with the prediction by Katz-Klemm-Vafa \cite{KKV},
we will also make a comparison of this result with the corresponding result
for counting curves in $\PP^2$.
In Section 6 we refine our approach to get stronger results not predicted
in the literature and make more observations on the general behavior of
the transformed Gopakumar-Vafa invariants of $\kappa_{\PP^2}$.

\vspace{.1in}
{\em Acknowledgements}.
The second author is partially supported by  NSFC grant 1171174.
The first author thanks Professor Gang Tian for helpful suggestions.

\section{Counting curves in algebraic surfaces}
\label{sec:Motivations1}

In this section we recall some results  on counting curves in a linear system on an algebraic surface,
especially on a K3 surface.
Even though we will not use the results and methods in this section,
the content of this section will suggest a suitable perspective
to understand our results.
We will see that the following three  features will emerge
also in the Calabi-Yau 3-fold setting that we will study in this paper:
Connection with quasimodular forms,
connection with Hilbert schemes,
quadratic property of node polynomials after some transformation.

\subsection{Counting curves in the projective plane}

 Let $N_d$ be the number of  irreducible rational curves in $\PP^2$
of degree $d$ through $3d-1$ points in general position.
Then by an application of the WDVV equation Kontsevich \cite{K-M} proved
the following recursive formula:
\ben
&& N_1 = 1, \\
&& N_d = \sum_{\substack{d_1+d_2 =d\\d_1,d_2\geq 1}} N_{d_1}N_{d_2}
\biggl[d_1^2d_2^2\binom{3d-4}{3d_1-2} - d_1^3d_2 \binom{3d-4}{3d_1-1}\biggr], \;\;\; d> 1.
\een
The genus formula states that an irreducible algebraic curve in $\PP^2$ of degree $d$
with only $\delta$ nodes as singularities has geometric genus
\be
g = \frac{(d-1)(d-2)}{2} - \delta.
\ee
So the irreducible rational curves of degree $d$ must have $g(d) = \frac{(d-1)(d-2)}{2}$ nodes.
For $0 \leq \delta \leq g(d)$,
denote by $N^{g(d) - \delta}_d$ the number of irreducible plane curves of degree $d$,
genus $g=g(d) - \delta$,
with $\delta$ simple nodes, through $3d-1+g$ points in general position.
For example \cite{Kle-Pie, Vai}:
\ben
N^{g(d)-1}_d & = & 3(d-1)^2, \;\; d \geq  3, \\
N^{g(d)-2}_d & = & \frac{3}{2}(d-1)(d-2)(3d^2-3d-11), \;\; d \geq 4, \\
N^{g(d)-3}_d & = & \frac{9}{2}d^6 - 27d^5 +\frac{9}{2}d^4
+\frac{423}{2}d^3 - 229d^2- \frac{829}{2}d \\
& + & 525 - \delta_{d,4}\binom{11}{2}, \;\; d \geq 4, \\
N^{g(d)-4}_d & = & \frac{27}{8}d^8 - 27d^7 + \frac{1809}{4}d^5 - 642d^4 - 2529d^3 \\
& + & \frac{37881}{8}d^2 + \frac{18057}{4}d - 8865 -\delta_{d,5}\binom{16}{2}, \;\;\; d \geq 5.
\een
The following observation was made in \cite{DiF-Itz}:
For fixed $\delta$,
when $d$ is large enough, $N^{g(d)-\delta}_d$ is a polynomial in $d$ of degree $2\delta$:
\ben
z_{\delta}(d) & = &
\frac{3^\delta}{\delta!} \biggl[d^{2\delta} - 2\delta d^{2\delta-1}
+ \frac{\delta(4 - \delta)}{3}d^{2\delta-2} \\
& + & \frac{\delta(\delta - 1)(20\delta - 13)}{6}d^{2\delta-3} \\
& - & \frac{\delta(\delta - 1)(69\delta^2 - 85\delta + 92)}{54}d^{2\delta-4}
+ \cdots \biggr].
\een
See Fomin-Mikhalkin \cite{F-M} for a proof.

\subsection{Counting rational nodal curves in K3 surfaces}

Let $C$ be a smooth curve in a generic K3 surface $X$ representing a primitive homology class,
with $C \cdot C = 2n-2$.
Then $C$ has genus $n$ and move in a complete linear system $|C| \cong \PP^n$.
By a result of Chen \cite{Che},
all the rational curves in  $|C|$ are nodal,
hence by the genus formula,
they all have $n$ nodes.
Denote by $N(n)$ the number of such curves in $|C|$.
Then Yau-Zaslow formula \cite{Yau-Zas} is
\be \label{eqn:YauZaslow}
\sum_{n \geq 0} N(n) q^{n-1} = \frac{1}{\Delta(q)}
= \frac{1}{q\prod_{n=1}^\infty (1-q^n)}.
\ee
Let us briefly outline the beautiful argument of Yau and Zaslow,
which has been completed into a mathematical proof by  \cite{Bea}
under suitable conditions to be specified below.
Look at the compactified universal Jacobian $\pi:  \cJ \to |C|$ for the linear system $|C|$.
Under the assumption that
If one assumes that all the curves in $|C|$ are reduced and irreducible,
$\cJ$ is a smooth hyperk\"ahler manifold of dimension $2n$,
and $\cJ$ is birationally equivalent to the Hilbert scheme $X^{[n]}$.
Under the  assumption
that each member in the linear system $|C|$ has at most nodal singularities,
then one can argue that
\be
N_0(n) = \chi(\cJ).
\ee
Now by a result of Batyrev \cite{Bat},
\be
\chi(\cJ) q^{n-1} =  \chi(X^{[n]}) q^{n-1}.
\ee
By a result of G\"ottsche \cite{Got},
\be
 \sum_{n=0}^\infty \chi(X^{[n]}) q^{n-1}
= \frac{1}{q\prod_{n=1}^{\infty} (1-q^n)^{\chi(X)}}= \frac{1}{\Delta(q)}.
\ee
Hence \eqref{eqn:YauZaslow} is obtained by combining  the three equalities above.

\subsection{Counting nodal curves on algebraic surfaces}

Let $L$ be a line bundle on a projective algebraic surface $S$.
Denote by  $t_\delta^S(L)$ of the numbers of $\delta$-nodal curves in a general $\delta$-dimensional
sub-linear system of $|L|$.
The following result was conjectured  by G\"ottsche \cite{Got}
and proved by Liu \cite{Liu} and Tzeng \cite{Tz}:
For every integer $\delta \geq  0$,
there exists a universal polynomial $T_\delta(x, y, z, t)$ of degree $\delta$ with the following property:
Given a smooth projective
surface S and a $(5\delta-1)$-very ample ($5$-very ample if $\delta = 1$) line bundle $L$ on S,
$$t_\delta^S(L) = T_r(L^2,L\kappa_S, c_1(S)^2, c_2(S)).$$
Kool-Shende-Thomas \cite{Koo-She-Tho} gave a different proof
and weakened the condition to $L$ being $\delta$-very ample.

Inspired by the Yau-Zaslow formula,
G\"ottsche \cite{Got} conjectured the following closed form of this generating function:
There exist universal power series $B_1(q)$ and $B_2(q)$ such that
\be \label{eqn:GYZ}
\begin{split}
& \sum_{\delta \geq 0}
T_\delta(L^2,L\cdot \kappa_S, c_1(S)^2, c_2(S)) \cdot (DG_2(q))^\delta \\
= & \frac{(DG_2(q)/q)^{\chi(L)}B_1(q)^{K^2_S} B_2(q)^{L\cdot \kappa_S}}
{(\Delta(q)D^2G_2(q)/q^2)^{\chi(\cO_S)/2}},
\end{split}
\ee
where $ D = q \frac{d}{d q}$,
$G_2$ is the second Eisenstein series
$$G_2(q) = -\frac{1}{24} + \sum_{n>0} \sum_{d|n} d \cdot q^n.$$
This formula is called the G\"ottsche-Yau-Zaslow formula.
See Liu \cite{Liu} and Tzeng \cite{Tz} for proofs.

For a K3 surface $S$,
G\"ottsche \cite{Got} conjectured:
\be
\sum_{l \in \Z} n^S_\delta(l) q^l
= \frac{(DG_2(q))^{\delta}}{\Delta(q)},
\ee
where for $L$ sufficiently ample $n^S_{\delta}(L^2/2)$ stands for
the number of $\chi(L) - \delta-1$-nodal curves in a $\delta$-dimensional sub-linear system of $L$.
See Bryan-Leung \cite{Bry-Leu} and Liu \cite{Liu} for proofs.
For Gromov-Witten theoretical approach to this problm,
see the work by Klemm {\em et al} \cite{KMPS}.

G\"ottsche \cite{Got} made a connection to Hilbert schemes of $S$
in a different way from that of Yau and Zaslow \cite{Yau-Zas}.
Let $Z_n(s) \subset S \times S^{[n]}$ be the universal family with projections $p_n: Z_n(S) \to S$,
$q_n: Z_n(S) \to S^{[n]}$.
Then for any line bundle $L$ on $S$,
$L_n:=(q_n)_*p_n^*L$ is a vector bundle of rank $n$ on $S^{[n]}$.
Let $S_2^\delta \subset S^{[3\delta]}$ be the closure of the locally closed subset
$$S_{2,0}^\delta
= \biggl\{ \coprod_{i=1}^\delta \Spec(\cO_{S,x_i}/m_{S,x_i}^2) \biggl| x_1, \dots, x_{\delta}
\text{\;\; are distinct points on $S$} \biggr\}.$$
G\"ottsche showed that when $L$ is $(5\delta-1)$-very ample
($5$-very ample if $\delta = 1$),
then a general $\delta$-dimensional sublinear system of $|L|$ contains exactly
\be
d_\delta(L) = \int_{S_2^\delta} c_{2\delta}(L_{3\delta})
\ee
curves with precisely $\delta$ nodes as singularities.
This connection is crucial to Tseng's proof of the G\"ottsche-Yau-Zaslow formula.
Another connection with Hilbert schemes was made in \cite{KST}.

\subsection{Quadratic property of node polynomials after transformation}
\label{sec:SurfacesQuad}

The Severi degree $N^{d, \delta}$ is the number of plane curves of degree $d$ with $\delta$ nodes
passing through $(d^2+3d)/2-\delta$ general points.
G\"ottsche \cite{Got} conjectured that for $\delta \leq 2d-2$,
$N^{d,\delta} = t_{\delta}^{\PP^2}(\cO(d))$.
He also conjectured for $d > 0$,
\be
\sum_{\delta} N^{d, \delta} x^\delta
= \exp (d^2 C_1(x) + d C_2(x) + C_3(x))
\ee
modulo the ideal generated by $x^{2d-1}$.
In particular,
for $\delta \leq 2d-2$,
the numbers $N^{d, \delta}$ are given by a polynomial $z_\delta(d)$ (called the node polynomials) of degree $2\delta$ in $d$,
as conjectured by DiFrancesco-Itzykson \cite{DiF-Itz}
and proved by Fomin-Mikhalkin \cite{F-M}.
Kleiman and Piene \cite{Kle-Pie2}  made a slightly different formulation of
G\"ottsche's conjecture as follows (here we use the version of Fomin-Mikkhalkin \cite{F-M}).
Define polynomials $A_j(d)$ by:
\be
\sum_{j=1}^\infty A_j(d) \frac{x^j}{j}
= \log \sum_{\delta} z_{\delta}(d) x^d.
\ee
Then $A_j(t)$ are quadratic polynomials with integral coefficients.
Kleiman and Piene \cite{Kle-Pie2} established this for $j \leq 8$:
\ben
&& A_1(d) = 3(d^2 -2d +1), \\
&& A_2(d) = - 3(14d^2 -39 d + 25), \\
&& A_3(d) = 3(230d^2 - 788d + 633), \\
&& A_4(d) = -9(1340d^2 - 5315d + 5023), \\
&& A_5(d) = 9(24192d^2 - 107294d + 114647), \\
&& A_6(d) = -9(445592d^2 - 2161292d + 2545325), \\
&& A_7(d) = 54(1386758d^2 - 7245004d + 9242081), \\
&& A_8(d) = -9(156931220d^2 - 873420627d + 1191950551).
\een
Recently,
building on ideas of Fomin and Mikhalkin \cite{F-M},
Block \cite{Blo} developed an explicit algorithm for computing node polynomials,
and used it to compute $N_\delta(d)$ for $\delta \leq 14$,
hence verified the above conjecture and found $A_j(d)$ up to $j = 14$.

Note the polynomials $A_1, A_2, \dots$ are obtained from the node polynomials $z_1, z_2, \dots$
by the following transformation:
\ben
&& A_1 = z_1, \\
&& A_2 = 2z_2 - z_1^2, \\
&& A_3 = 3z_3 - 3z_1z_2 + z_1^3,
\een
in general,
\ben
A_n = n \sum_{\substack{m_1, \cdots, m_k \geq 0 \\ \sum_j jm_j = n}}
(-1)^{m_1+\cdots + m_k-1} \frac{(\sum_j m_j-1)!}{m_1!\cdots m_k!} z_1^{m_1} \cdots z_k^{m_k}.
\een
Note the coefficients on the right-hand side are all integers.

\section{Counting Curves in Calabi-Yau 3-folds}
\label{sec:Motivations2}

In this section we first recall in \S \ref{sec:Quintic} some results for curve counting
in quintic Calabi-Yau 3-fold.
Then we review the definition of the Gopakumar-Vafa invariants in \S \ref{sec:GV},
and review in \S \ref{sec:Transform} the observations made in \cite{zhou1}
about the quadratic property mentioned in the Introduction of the Gopakumar-Vafa invariants
under a suitable transformation.

\subsection{Gromov-Witten invariants of  quintic 3-fold} \label{sec:Quintic}

Let $X_5$ be the quintic 3-fold,
and let
\ben
K^g_d = \int_{[\Mbar_{g,0}(X_5;d)]} 1
\een
be degree $d$ genus $g$ Gromov-Witten invariants of $X$.
The calculations of such invariants are very difficult
and predictions made by physicists have played a crucial role
in the mathematical research on them.
Let
\be
\sum_{j=0}^\infty I_j(t) w^j
= e^{wt} \sum_{d=0}^\infty e^{dt} \frac{\prod_{r=1}^{5d} (5w+r)}{\prod_{r=1}^d (w+r)^5}.
\ee
For example,
\ben
I_0(t) & = & 1 + \sum_{d =1}^\infty e^{dt} \frac{(5d)!}{(d!)^5}, \\
I_1(t) & = & t I_0(t)  + 5 \sum_{d=1}^\infty e^{dt} \frac{(5d)!}{(d!)^5}
\sum_{r=d+1}^{5d} \frac{1}{r}, \\
I_2(t) &  = & \frac{t^2}{2}I_0(t) + 5 t \sum_{d=1}^\infty e^{dt} \frac{(5d)!}{(d!)^5}
\sum_{r=d+1}^{5d} \frac{1}{r} \\
& + & \frac{1}{2}\sum_{d=1}^\infty e^{dt} \frac{(5d)!}{(d!)^5}
\biggl( 25 \biggl(\sum_{r=d+1}^{5d} \frac{1}{r} \biggr)^2
- \sum_{r=1}^{5d} \frac{25}{d^2} + \sum_{r=1}^{d} \frac{5}{r^2} \biggr).
\een
By the Frobenius method,
it is easy to see that
$\{I_0, I_1, I_2, I_3\}$ form a basis of solutions to the Picard-Fuchs equation
\ben
\pd_t^4 f = 5e^t (5\pd_t+1) \cdots (5\pd_t +4) f(t).
\een
Let
\be
J_j(t) = \frac{I_j(t)}{I_0(t)},
\ee
and
\be
T = J_1(t).
\ee
Givental \cite{Giv} and Lian-Liu-Yau \cite{LLY} proved the following prediction
due to Candelas-de la Ossa-Green-Parkes \cite{CDGP}:
\be
F_0(T) = \frac{5}{2} ( J_1(t) \cdot J_2(t) - J_3(t) ).
\ee
Zinger \cite{Zin} proved the following prediction due to Bershadsky-Cecotti-Ooguri-Vafa \cite{BCOV}:
\be
F_1(T)
= \frac{25}{12}(J_1(t) - t) - \log \biggl( I_0(t)^{31/3}(1-5^5e^t)^{1/12}J_1'(t)^{1/2} \biggr).
\ee
Predictions up to $g = 51$ have been made by
Huang-Klemm-Quackenbush \cite{HKQ}
based on ideas in \cite{BCOV} and the work of Yamaguchi and Yau \cite{Y-Y}.

\subsection{From GW invariants to curve counting in Calabi-Yau 3-folds}
\label{sec:GV}

In general,
the numbers $K^0_d$ are rational numbers.
Candelas {\em et al} \cite{CDGP} suggested to consider numbers $n_d^0$ defined as follows:
\ben
K^0_d = \sum_{k|d} n_k \frac{k^3}{d^3},
\een
i.e.,
\ben
\sum_{d \geq 1} K^0_d e^{dT} = \sum_{k \geq 1} n_k^0 \sum_{l=1}^\infty \frac{e^{klT}}{l^3}.
\een
Then one can observe that $n_k^0$ become integral, for example,
$n^0_1 = 2875$, $n^0_2 = 609250$.
These numbers are expected to be the ``number" of rational curves of degree $k$ in $X_5$.
In general,
Gopakumar and Vafa \cite{GV} found a relationship between
the Gromov-Witten invariants of a Calabi-Yau 3-fold $X$
and counts of BPS states in M-theory.
This suggests the definition of some invariants
called the Gopakumar-Vafa BPS invariants which we now turn to.
They will be used as alternatives of ``number" of curves in $X$.
We will recall their definition momentarily.

Let $X$ be a projective algebraic variety.
For any $\beta\in H_2(X,\Z)$,
denote by $\Mbar_{g,n}(X,\beta)$ the moduli space of stable maps
from genus $g$ $n-$pointed curves to $X$ with the image in the class $\beta$ :
 $$
 f:(C,p_1,\cdots,p_n)\rightarrow X \quad \text{stable map,  s.t.} \quad f_*[C]=\beta.
 $$
The virtual dimension of this space is
$$\vdim \Mbar_{g,n}(X,\beta) =(1-g)(\dim X-3)+\int_\beta c_1(TX)+n.$$
When $X$ is a Calabi-Yau $3-$fold,
the virtual dimension of $\Mbar_{g,0}(X;\beta)$ is $0$.
Define the $0$-point Gromov-Witten invariant $K^g_\beta(X)$ by
$$K^g_\beta(X):=\int_{[\Mbar_{g}(X, \beta)]^{vir}} 1.
$$
\begin{conjdefn}
There are integers $n^g_\beta(X)$ called Gopakumar-Vafa BPS invariants such that:
\be \label{def:GV}
\sum_{\beta\neq 0}\sum_{g \geq 0} K^g_\beta (X) \lambda^{2g-2} t^\beta
=\sum_{\beta\neq 0}\sum_{g \geq 0} n^g_\beta (X) \sum_{k>0} \frac{1}{k}(2\sin(\frac{k\lambda}{2}))^{2g-2}t^{k\beta},
\ee
where $t^{k\beta}$ is an element of the Novikov ring associated
with $H_2(X;\Z)$,
i.e. the group ring of $H_2(X;\Z)$.
\end{conjdefn}

For fixed genus $g \geq 0$,
denote by
\be
F_g^X : = \sum_{\beta\neq 0}  K^g_\beta (X)   t^\beta
\ee
the instanton part of the genus $g$ free energy.
By taking the coefficients of $\lambda^{-2}$ on both sides of \eqref{def:GV},
one gets
\be
F_0^X = \sum_{\beta\neq 0}  n^0_\beta (X) \sum_{k>0} \frac{1}{k^3} t^{k\beta}
= \sum_{\beta \neq 0} n^0_\beta(X) \cdot  \Li_3(t^\beta),
\ee
where $\Li_r(z) = \sum_{k > 0} k^{-r} z^k$.
By taking the coefficients of $\lambda^0$ on both sides of \eqref{def:GV},
one gets
\be
\begin{split}
F_1^X = & \sum_{\beta\neq 0}  n^0_\beta (X) \sum_{k>0} \frac{1}{12k} t^{k\beta}
+  \sum_{\beta\neq 0}  n^1_\beta (X) \sum_{k>0} \frac{1}{k} t^{k\beta} \\
= & \sum_{\beta \neq 0} (\frac{1}{12} n^0_\beta(X) + n^1_\beta(X)) \cdot  \Li_1(t^\beta).
\end{split}
\ee
The numbers $n^1_\beta(X)$ should not be interpreted as the ``numbers" $n^{*1}_\beta(X)$ of elliptic curves in class $\beta$,
but instead the number of BPS states associated to $\beta$.
By a result of Pandharipande \cite{Pan},
they are related as follows:
\be
\sum_{\beta \neq 0} \sum_{k > 0} \frac{n^1_\beta(X)}{k}t^{k\beta}
= \sum_{\beta \neq 0} \sum_{k > 0} n^{*1}_\beta \frac{\sigma_1(k)}{k} t^{k\beta},
\ee
where $\sigma_r(k) = \sum_{l|k} l^r$.
Combining this with the above formula for $F_1^X$,
one gets
\be
F_1^X = \sum_{\beta \neq 0} (\frac{1}{12} n^0_\beta(X) \cdot \Li_1(t^\beta)
+ n^{*1}_\beta(X) \cdot  \log \prod_{l=1}^\infty (1-t^{l\beta})^{-1}).
\ee
For more mathematical work on contributions of embedded curves to the Gromov-Witten invariants,
see Faber-Pandharipande \cite{FP}, Bryan-Pandharipande \cite{BP}, Bryan-Leung \cite{Bry-Leu}, and Bryan \cite{Bry}.

\subsection{Local Gromov-Witten invariants and local Gopakumar-Vafa invariants}

A powerful technique for the computations of
Gromov-Witten invariants is the localization method \cite{K, Gra-Pan}.
The successes in computing the genus zero and genus one Gromov-Witten
invariants of quintic 3-fold are based on transforming the calculations to
the ambient space $\PP^4$.
Unfortunately it is not clear how this can be done for genus $g > 1$ at present.

While a compact Calai-Yau 3-fold does not have continuous symmetry group,
a noncompact Calabi-Yau 3-fold may have a 3-torus as symmetry group
which has isolated fixed points.
For example,
let $S$ be a toric Fano surface and let $X=\kappa_S$ be the total space
of its canonical line bundle.
It is interesting to study their Gromov-Witten invariants of such spaces \cite{CKYZ,KZ}.
First of all,
suppose that $S$ is a toric Fano surface embedded in a Calabi-Yau 3-fold $Y$.
Because $S$ is Fano, its canonical line bundle is negative,
so the normal bundle $N_{S/Y}$ is negative because it is isomorphic
to $\kappa_S$ by the adjunction formula and the triviality of $K_Y$.
If we have a stable map $f:C \to Y$,
such that $f(C) \subset S$,
by the negativity of $N_{S/Y}$,
under deformations of $f$,
the images will remain in $S$.
In other words,
let $\beta \subset H_2(S;\Z)$,
and let $i_S: S \to Y$ be the inclusion,
then $\Mbar_{g,0}(S; (i_S)_*\beta)$ will have a component
isomorphic to $\Mbar_{g,0}(S;\beta)$,
so so there is a contribution, denoted by $K^g_\beta(\kappa_S)$,
to the Gromov-Witten invariant $K^g_{(i_S)*\beta}(Y)$,
from this component.
Recall we have the following diagram
$$
\xymatrix{
  \overline M_{g,1}(S, \beta) \ar[d]_{ev} \ar[r]^{\pi} &     \overline M_{g}(S, \beta)   \\
  S                     }$$
where $\pi$ is the forgetful map and $ev$ is the evaluation map.
Define a bundle $(\kappa_S)^g_\beta$ as follows:
$$
(\kappa_S)^g_\beta: = \mathcal R^1\pi_* ev^* \kappa_S,
$$
One can show that
$$
K^g_\beta(\kappa_S)=\int_{[\Mbar_{g,0}(S, \beta)]^{vir}} e((\kappa_S)^g_\beta).
$$

Besides providing local contributions to the Gromov-Witten invariants,
another more important motivation to study local Gromov-Witten invariants
is that by suitably choosing the local Calabi-Yau geometries,
one can reproduce the partition functions of gauge theories,
an idea called the geometric engineering by the string theorists \cite{KKV1}.

In exactly the same fashion as \eqref{def:GV},
the local Gopakumar-Vafa invariants of $\kappa_S$ is defined as follows:
\be \label{def:LocalGV}
\sum_{\beta\neq 0}\sum_{g \geq 0} K^g_\beta (\kappa_S) \lambda^{2g-2}t^\beta
=\sum_{\beta\neq 0}\sum_{g \geq 0} n^g_\beta (\kappa_S)
\sum_{k>0} (2\sin(\frac{k\lambda}{2}))^{2g-2}\frac{t^{k\beta}}{k}.
\ee
Suppose that $C$ is a smooth curve in $S$
such that $[C] = \beta$,
then  the genus of $C$ which we denote by $g(\beta)$ is given by the adjunction formula:
\be
g(\beta) = 1 + \frac{1}{2} (\beta^2 + \kappa_S \cdot \beta).
\ee
Based on empirical evidence in \cite{AMV},
it is implicitly assumed in physics literature that
\be
n^g_\beta(\kappa_s) = 0
\ee
for $g> g(\beta)$.
We will prove this in later sections.
Under this assumption,
the summation over $g$ in \eqref{eqn:LocalGV2} becomes a finite summation $\sum_{g=0}^{g(\beta)}$.

\subsection{Quadratic properties of the local Gopakumar-Vafa invariants}
\label{sec:Transform}

Now we recall an observation made by the second author in \cite{zhou1}.
Write the left-hand side of \eqref{def:LocalGV} as $F^{\kappa_S}$ and
let $q=e^{\sqrt{-1} \lambda}$,
then one can rewrite \eqref{def:LocalGV} in the following form
\be \label{eqn:LocalGV2}
F^{\kappa_S}=\sum_{\beta\neq 0}\sum_{g \geq 0} n^g_\beta (\kappa_S)
\sum_{k>0} (-1)^{g-1}(q^{k/2}-q^{-k/2})^{2g-2} \frac{t^{k\beta}}{k}.
\ee
Let us recall how Gopakumar and Vafa \cite[Section 2]{GV} defined the number $n^g_\beta(X)$.
See also Katz-Klemm-Vafa \cite[Section 3]{KKV}, especially (3.1)-(3.3).
Then for each $\beta \in H_2(X; \Z)$, and half integers $j_L, j_R \in \half \Z$,
the number $N^\beta_{j_L,j_R}(X)$ denote the number of BPS states with charge represented by the class
$\beta$ and with $SU(2)_L\times SU(2)_R$ representation $[(j_L)] \otimes [(j_R)]$ indexed by $(j_L, j_R)$.
The numbers $N_{j_L,j_R}^\beta(X)$ may change under the deformation of the complex structures
on $X$,
however, they claimed that
the number
$$\sum_{j_R} (-1)^{2j_R} (2j_R+1) N^\beta_{j_L,j_R}$$
is an invariant,
and they defined the invariants $n^g_\beta(X)$ by (cf. \cite[(3.2)]{KKV}):
\be \label{eqn:N-n}
\sum_g n^g_\beta(X) I_g  = \sum_{j_L} (\sum_{j_R} (-1)^{2j_R} (2j_R+1) N^\beta_{j_L,j_R} ) [(j_L)],
\ee
where
\be
I_g = [(\half)+2(0)]^{\otimes g}.
\ee
Let us give some explanation of the physical notations.
By $SU(2)_L$ and $SU(2)_R$ we mean two copies of $SU(2)$.
Let
\be
t_{L, R} = \begin{pmatrix} e^{i\theta} & 0 \\ 0 & e^{- i \theta} \end{pmatrix}.
\ee
By $[(j_{L,R})]$ we mean the irreducible representation of $SU(2)_{L,R}$ on which $t_{J,R}$
acts as
$$\diag(e^{2j_{L,R} i\theta}, e^{2(j_{L,R}-1) i \theta}, \dots, e^{-2j_{L,R}i\theta})$$
in a suitable basis. This representation has dimension $2j_{L,R}+1$.
Given a representation $V$ of $SU(2)$,
its character $\chi_V$ is the trace of the matrix representation of $t$.
In particular,
\bea
&& \chi_{ [(j)]} = e^{2j i\theta} + e^{2(j-1) i \theta} + \dots + e^{-2ji\theta}, \\
&& \chi_{ I_g} = (e^{i\theta} + e^{- i \theta} + 2)^g.
\eea
Write $e^{i\theta} = -q$,
one gets:
\bea
&& \chi_{ [(j)]} = (-1)^{2j} (q^{2j} + q^{2(j-1)} + \dots + q^{-2j}), \\
&& \chi_{ I_g} = (-1)^g (q^{1/2} - q^{- 1/2})^{2g}.
\eea
Write $R_g(q) = q^g + q^{g-2} + \cdots + q^{-g}$.
Then by \eqref{eqn:N-n},
\ben
&& \sum_g n^g_\beta(X) (-1)^g (q^{1/2} - q^{- 1/2})^{2g} \\
& = & \sum_{j_L} (\sum_{j_R} (-1)^{2j_L+2j_R} (2j_R+1) N^\beta_{j_L,j_R}(X) ) R_{2j_L}(q).
\een
Define
\be
N^g_\beta(X) = \sum_{j_R} (-1)^{g+2j_R} (2j_R+1) N^\beta_{g/2,j_R}(X).
\ee
We then arrive at the following identity (cf. \cite[(111)]{HIV},
see also \cite[(5)]{zhou1}):
\be
\sum_g n^g_\beta(X) (-1)^g (q^{1/2} - q^{- 1/2})^{2g}
= \sum_{g \geq 0} N^g_\beta(X) R_g(q).
\ee
The main observation of \cite{zhou1} is the following:

\begin{conjecture}
Fix $\delta \geq 0$.
For local Calabi-Yau geometries given by the canonical line bundles of toric Fano surfaces,
$N^{g(\beta)-\delta}_\beta(\kappa_S)$ is a quadratic polynomial in $d$, up to a suitable sign,
when $d$ is sufficiently large compared to $\delta$.
\end{conjecture}

It was also suggested in \cite{zhou1} that the argument in \cite{KKV} might give an explanation
of quadratic behavior of the transformed Gopakumar-Vafa invariants $N^g_\beta$.

Note the transformation from the invariants $n^g_\beta$ to $N^g_\beta$ is quite different
from the transformation of the node polynomials introduced by G\"ottsche \cite{Got}
(cf. \S \ref{sec:SurfacesQuad}).

\section{Some Combinatorial Preliminaries}

In this section we collect some  combinatorial results
which we use later.

\subsection{Partitions and some associated combinatoiral numbers}
Let $\mu=(\mu_1,\mu_2,\dots,\mu_l, \dots)$ be a partition,
i.e.,
$\mu_1 \geq \mu_2 \geq \cdots$ is a sequence of nonnegative integers
such that $\mu_n = 0$ for $n \gg 0$.
We will use the following numbers associated with $\mu$:
\ben
&& l(\mu)= |\{i:\,\mu_I>0\}|, \quad
|\mu|=\sum_{i\geq 1} \mu_i, \quad m_i(\mu)=|\{i:\,\mu_j=i\}|,\\
&& z_\mu = \prod_{i\geq 1} m_i(\mu)! \prod_{i\geq 1} \lambda_i,  \quad
 n_\mu=\sum_{i \geq 1}(i-1) \mu_i, \\
&& k_\mu= n_{\mu^t}-n_\mu=\sum_{i\geq 1} \mu_i(\mu_i-2i+1).
\een
It is very useful to represent a partition $\mu$ by its Young diagram $Y(\mu)$.
By transposing the rows and columns of $Y(\mu)$,
we get the Young diagram of another partition, denoted by $\mu^t$.
For a box $x$ sitting at the $i$-th row and the $j$-th column of $Y(\mu)$,
its arm number, leg number and hook number are defined by
\begin{align}
a(x) & = \mu_i - i, & l(x) & = \mu^t_j - j, & h(x) & = a(x) + l(x) + 1.
\end{align}

Let $\mu$ be a partition of $n>0$.
Then $\mu$ can be used either to index an irreducible representation $R_\mu$
of $S_n$,
or a conjugacy class $C_\mu$ of $S_n$ (cf. e.g. \cite[\S 1.7]{MacDonald}).
In particular,
$R_{(n)}$ is the trivial representation,
$R_{(1^n)}$ is the sign representation;
$C_{(1^n)}$ is the conjugacy class of the identity element,
and $C_{(n)}$ is the conjugacy class of the $n$-cycles.
For two partitions $\mu$ and $\nu$,
denote by $\chi_\mu(\nu)$  the value of the irreducible character $\chi_{R_\mu}$
on the conjugacy class $C_{\nu}$.
They satisfy the following orthogonality relations:
\bea
&& \sum_{\nu} \frac{1}{z_\nu} \chi_{\mu^1}(\nu) \cdot \chi_{\mu^2}(\nu)
= \delta_{\mu^1, \mu^2}, \\
&& \sum_{\mu} \chi_{\mu}(\nu^1) \cdot \chi_{\mu}(\nu^2)
= \delta_{\nu^1, \nu^2} z_{\nu^1}.
\eea

\subsection{Estimates of some combinatorial numbers}

Let $\mu$ and $\nu$ be two partitions of $n$.
We write $\mu> \nu$ or $\nu < \mu$
if the first nonzero $\mu_i - \nu_i$ is positive.
This defines an ordering on the set $\cP_n$ of partitions of $n$,
called the {\em reverse lexicographic ordering} \cite{MacDonald}.

\begin{lem} \label{lm:NKappaEst}
Suppose that $\mu, \nu \in \cP_n$ and $\mu > \nu$,
then one has
\be
 n_\mu < n_\nu, \quad \kappa_\mu > \kappa_\nu.
\ee
In particular,
if $\mu < (n)$,
then
\be
\kappa_\mu \leq \kappa_{(n)} - 2n.
\ee
\end{lem}

\begin{proof}
It suffices to consider the case when $\mu$ and $\nu$ are adjacent
in the reverse lexicographic ordering.
Suppose that $l(\mu) = l$.
We have two cases to consider. \\
Case 1. Suppose that
$\mu_l > 1$,
then one has
$$\nu = (\mu_1, \dots, \mu_{l-1}, \mu_l -1, 1).$$
Then we have
\ben
n_\mu - n_\nu
& = & \sum_{i=1}^{l-1} (i-1) \mu_i + (l-1)\mu_l \\
& - & (\sum_{i=1}^{l-1} (i-1) \mu_i + (l-1)  (\mu_l-1) + l \cdot 1) \\
& = & -1 < 0,
\een
and
\ben
\kappa_\mu - \kappa_\nu
& = & \sum_{i=1}^{l-1} \mu_i(\mu_i-2i+1) + \mu_l(\mu_l -2l+1) \\
& - & \big( \sum_{i=1}^{l-1} \mu_i(\mu_i-2i+1)
+ (\mu_l-1)(\mu_l-1 -2l+1) \\
&& + 1 \cdot (1 - 2(l+1)+1) \big) \\
& = &  \mu_l(\mu_l -2l+1)
- (\mu_l-1)(\mu_l-2l) +2l \\
& = & 2\mu_l > 0.
\een
Case 2. Suppose that $\mu_k>1$, $\mu_{k+1} = \mu_{k+2} = \cdots = \mu_l = 1$.
Then we have
$$\nu =(\mu_1, \dots, \mu_{k-1}, \mu_k - 1, \nu_{k+1} = 1, \dots, \nu_{l+1} = 1),$$
and so we have
\ben
n_\mu - n_\nu
& = & \sum_{i=1}^k (i-1)\mu_i + \sum_{j=k+1}^l (j-1) \cdot 1 \\
& - & (\sum_{i=1}^{k-1} \mu_i + (k-1) \cdot (\mu_k-1)
+ \sum_{j=k+1}^l (j-1) \cdot 1 + l \cdot 1 ) \\
& = & - (l-k+1) < 0
\een
and
\ben
\kappa_\mu - \kappa_\nu
& = & \sum_{i=1}^{k} \mu_i(\mu_i-2i+1)
+ \sum_{j=k+1}^l 1 \cdot (1 - 2j+1) \\
& - & \big( \sum_{i=1}^{k-1} \mu_i(\mu_i-2i+1)
+ (\mu_k-1)(\mu_k-1 -2k+1) \\
&& + \sum_{j=k+1}^{l+1} 1 \cdot (1 - 2j+1) \big) \\
& = &  \mu_k(\mu_k -2k+1)
- (\mu_k-1)(\mu_k-2k) +2l \\
& = & 2\mu_k -2k + 2l > 0.
\een
\end{proof}

\subsection{Symmetric functions and their specializations}

Let $x=(x_1,x_2, \dots)$ be a sequence of indeterminates.
Consider the infinite product
$\prod_{i=1}^\infty (1-x_it)$.
One has the following expansions:
\bea
&& \prod_{i=1}^\infty (1-x_it) = \sum_{k=0}^\infty (-1)^k e_k(x) t^k, \\
&& \prod_{i=1}^\infty \frac{1}{(1-x_it)} =  \sum_{k=0}^\infty h_k(x) t^k, \\
&& \log \prod_{i=1}^\infty (1-x_it) = - \sum_{k=0}^\infty p_k(x) \frac{t^k}{k},
\eea
where
\ben
&& e_k(x)=\sum_{i_1<i_2<\cdots< i_k}x_{i_1}x_{i_2}\cdots x_{i_k},\\
&& h_k(x)=\sum_{\sum k_j=k} x_{i_1}^{k_1}x_{i_2}^{k_2}\cdots x_{i_s}^{k_s}, \\
&& p_k(x)=\sum_{j \geq 1} x_j^k,
\een
These date back to Euler.
The following monomials of these symmetric functions can be used to obtain
additive basis of the space $\Lambda$ of symmetric functions:
\ben
e_\mu(x)=\prod_{i\geq 1} e_{\mu_i},\quad   h_\mu(x)=\prod_{i\geq 1} h_{\mu_i},\quad
p_\mu(x)=\prod_{i\geq 1} p_{\mu_i},
\een
so do the Schur functions $s_\mu$ defined by:
\ben
s_\mu(x)=\det(h_{\mu_i-i+j}(x))_{1\leq i,j \leq n},
\een
where $n > \max\{\mu_1, l(\mu)\}$.
The bases $\{s_\nu\}$ and $\{p_\eta\}$ are related as follows \cite[\S I.7]{MacDonald}:
\bea
&& s_{\nu}(x)
=\sum_{|\eta|=|\nu|}\frac{\chi_{\nu}(\eta)}{z_{\eta}} \cdot p_{\eta}(x), \label{eqn:SchurNewton1} \\
&& p_{\eta}(x)
=\sum_{|\eta|=|\nu|} \chi_{\nu}(\eta) \cdot s_{\nu}(x). \label{eqn:SchurNewton2}
\eea

\subsection{Symmetric functions and their specializations}
The following identities are also due to Euler:
\be \label{eqn:EulerEProduct}
\prod_{n=0}^\infty (1 - q^n z)
= 1+ \sum_{n=1}^\infty (-1)^n \frac{q^{n(n-1)/2}}{\prod_{j=1}^n (1-q^j)} z^n,
\ee
\be \label{eqn:EulerHProduct}
\prod_{n=0}^\infty \frac{1}{1-q^nz}
= 1+ \sum_{n \geq 1} \frac{1}{\prod_{j=1}^n (1-q^j)} z^n.
\ee
From these one gets:
\bea
&& e_k(1,q^{-1}, q^{-2}, \dots) =  \frac{q^{-k(k-1)/2}}{[k]_{q^{-1}}!}, \\
&& h_k(1,q^{-1}, q^{-2}, \dots) =   \frac{1}{[k]_{q^{-1}}!},
\eea
where we have  used the following notations:
\be
[k]_q:=1-q^{j}, \quad [k]_q!:=[1]_q [2]_q \cdots [k]_q.
\ee
It is trivial to see that
\be
p_k(1,q^{-1}, q^{-2}, \dots) = \frac{1}{1-q^{-k}} = \frac{1}{[k]_{q^{-1}}}.
\ee
We have \cite{MacDonald}:
\be \label{eqn:SchurSpec}
s_{\mu}(1,q^{-1}, q^{-2}, \dots)
=q^{-n(\mu)}\prod_{x\in Y(\mu)} \frac{1}{1-q^{-h(x)}},
\ee
where $Y(\mu)$ is the Young tableau of $\mu$ and $h(x)$ is the hook number
of $x$.

\subsection{Basic hypergeometric series}

First we recall the following notations:
\ben
&& (a;q)_\infty = \prod_{i=0}^{\infty}(1-aq^i), \\
&& (a;q)_n = \frac{(a;q)_\infty}{(aq^n;q)_\infty}
= \prod_{i=0}^{n-1} (1-aq^i).
\een
Let
\ben
&[\infty]_q!=(q;q)_{\infty}.
\een
The following Lemma will be used repeatedly in this work
(cf.  e.g. \cite{GR}, page 7-10).

\begin{lem} We have the $q$-analog of binomial theorem:
\be\label{binomial}
\sum_{n\geq 0} \frac{(a;q)_n}{(q,q)_n} z^n = \frac{(az;q)_\infty}{(z;q)_\infty}.
\ee
\end{lem}

A particular case of \eqref{binomial} is when $a=0$,
we get the $q$-analog exponential function:
\be\label{exp}
\sum_{n\geq 0} \frac{z^n}{[n]_q!}=\frac{1}{(z,q)_{\infty}}.
\ee
This is just another way to write \eqref{eqn:EulerHProduct}.

\subsection{Link invariants of the Hopf link}

For calculations of the Gromov-Witten invariants of
local Calabi-Yau geometries,
we will need the following expressions
arising from the colored HOMFLY polynomials
of the Hopf link (see e.g. \cite{AMV, I, zhou2}):
\be\label{w}
\W_{\mu^1,\mu^2}(q)=s_{\mu^1}(q^{\varrho})
s_{\mu^2}(q^{\mu^1+\varrho}),
\ee
where $q^\varrho=(q^{-1/2},q^{-3/2},q^{-5/2}, \cdots)$ and
$q^{\mu+\varrho}=(q^{\mu_1-1/2},q^{\mu_2-3/2},q^{\mu_3-5/2},\cdots)$.
These invariants have the following symmetric properties:
\be \label{W:Symmetry}
\W_{\mu,\nu}(q) = \W_{\nu,\mu}(q).
\ee
The following are some examples:
\begin{eqnarray*}
&& \W_{(1), (1)}(q) = \frac{q^2-q+1}{(q- 1)^2}, \\
&& \W_{(2), (1)}(q) = \frac{q^{\frac{3}{2}}(q^3-q^2+1)}{(q-1)^2(q^2-1)}, \\
&& \W_{(1,1), (1)}(q) = \frac{q^3-q+1}{q^{\frac{1}{2}}(q-1)^2 (q^2-1)}, \\
&& \W_{(2), (1,1)}(q) = \frac{q(q^4-q^2+1)}{(q-1)^2(q^2-1)^2}, \\
&& \W_{(1,1), (1,1)}(q)
= \frac{q^6-q^5+2q^3-q^2-q+1}{q^2(q-1)^2(q^2-1)^2}, \\
&& \W_{(2), (2)}(q) = \frac{q^2(q^6-q^5-q^4+2q^3-q+1)}{(q-1)^2(q^2-1)^2}, \\
\end{eqnarray*}
From these examples,
one can see that
$\W_{\mu\nu}$ has the following form:
$$\W_{\mu,\nu}(q) = \frac{f(q)}{g(q)} \;\;\text{or} \;\; \frac{q^{1/2} f(q)}{g(q)},$$
where $f(q)$ and $g(q)$ are polynomials in $q$.

In general,
if $f(q)$ is a series in $\C[q,q^{-1}]]$, i.e., of the form:
$$\sum_{n \leq M} a_n q^n$$
for some integer $M$ such that $a_M \neq 0$,
then we define
$\deg_q f(q) = M$.
It is easy to see that these definitions are well-defined.
In the following when we take $\deg_q$ of some function $f(q)$,
we mean first taking the expansion of $f(q)$ into a series in $\C[q,q^{-1}]]$.
If $f(q)$ and $g(q)$ are two series in $q$ of the above form,
then it is easy to see that
\ben
&& \deg_q \frac{f(q)}{g(q)} = \deg_q f(q) - \deg_q g(q).
\een
We define
\ben
&& \deg_q \frac{q^{1/2}f(q)}{g(q)} = \frac{1}{2} + \deg_q f(q) - \deg_q g(q).
\een

\begin{lem}\label{lm:WEst}
For two partitions $\mu$ and $\nu$,
one has
\be
\deg_q \W_{\mu,\nu}(q) \leq |\mu| \cdot |\nu| - (|\mu|+|\nu|)/2.
\ee
The equality holds iff $\mu=(|\mu|)$ and $\nu = (|\nu|)$.

Moreover, if $\mu\neq(|\mu|)$, we have
$$
\deg_q \W_{\mu,\nu}(q) \leq |\mu| \cdot |\nu| - (|\mu|+|\nu|)/2-(|\nu|+1),
$$
the equality holds iff $\mu = (|\mu|-1,1)$ and $\nu = (|\nu|)$;
if $\mu\neq(|\mu|)$ and $\mu\neq(|\mu|-1,1)$ , we have
$$
\deg_q \W_{\mu,\nu}(q) \leq |\mu| \cdot |\nu| - (|\mu|+|\nu|)/2-2(|\nu|+1),
$$
the equality holds iff $\mu = (|\mu|-2,2)$ and $\nu = (|\nu|)$.
\end{lem}

\begin{proof}
Because \eqref{eqn:SchurSpec}, we have
\be
s_{\mu}(q^{\varrho})=q^{-|\mu|/2-n(\mu)}\prod_{x\in Y(\mu)} \frac{1}{1-q^{-h(x)}},\label{sh1}
\ee
where $Y(\mu)$ is the Young tableau of $\mu$ and $h(x)$ is the hook number.
It is clear that
\be
\deg_q s_{\mu}(q^\rho) =-|\mu|/2 - n(\mu).
\ee
Note
\ben
p_m(q^{\mu+\rho})
& = & q^{-m/2} \biggl(q^{m\mu_1} + q^{m(\mu_2-1)} + \cdots
+ q^{m(\mu_l-(l-1))} + \sum_{j=l}^\infty q^{-mj} \biggr) \\
& = &  q^{-m/2} \biggl( q^{m\mu_1} + q^{m(\mu_2-1)} + \cdots
+ q^{m(\mu_l-(l-1))} + \frac{q^{-ml}}{1-q^{-m}} \biggr),
\een
so one has for all partitions $\eta$ with $|\eta| = |\nu|$,
\be
 p_{\eta}(q^{\mu+\rho}) = q^{|\nu| \cdot \mu_1 - |\nu|/2} + \cdots,
\ee
where $\cdots$ stands for lower order terms.
Hence by \eqref{eqn:SchurNewton1}:
\ben
s_{\nu}(q^{\mu+\varrho})
& = & \sum_{|\eta|=|\nu|} \frac{\chi_{\nu}(\eta)}{z_\eta} p_\eta(q^{\mu+\rho}) \\
& = & \sum_{|\eta|=|\nu|} \frac{\chi_{\nu}(\eta)}{z_\eta} q^{|\nu|\cdot \mu_1} + \cdots \\
& = & \delta_{\nu, (|\nu|)} q^{|\nu|\cdot \mu_1} + \cdots.
\een
In the last equality we have used the orthogonality relations for characters to get:
$$\sum_{|\eta| = |\nu|} \frac{\chi_{\nu}(\eta)}{z_\eta}
= \delta_{\nu, (|\nu|)}.$$
It follows that
\be
\deg_q s_{\nu}(q^{\mu+\rho}) \leq \mu_1 \cdot |\nu|-|\nu|/2
= \deg_q s_{(|\nu|)}(q^{\mu+\rho}),
\ee
with the equality holds iff $\eta =(|\eta|)$.
Therefore,
  by \eqref{w},
\be \label{eqn:Degree}
\deg_q \W_{\mu,\nu} \leq \mu_1 \cdot |\nu| - n(\mu)
- \half (|\mu|+|\nu|)
= \deg_q \W_{\mu,(|\nu|)}.
\ee
The  equality holds iff $\eta=(|\eta|)$.
Now we use the symmetry \eqref{W:Symmetry} to get:
\ben
\deg_q \W_{\mu,(|\nu||)}
& = & \deg_q \W_{(|\nu|),\mu} \leq \deg_q \W_{(|\nu|),(\mu|)} \\
& = & |\mu| \cdot |\nu| - (|\mu|+|\nu|)/2.
\een
The equality holds iff $\mu = (|\mu|)$.

If $\mu \neq (|\mu|)$, we have $\mu_1 \leq |\mu|-1$ and  $-n_\mu \leq -1$,  hence
\be
\deg_q \W_{\mu,\nu} \leq (|\mu|-1)\cdot |\nu| - 1
- \half (|\mu|+|\nu|)
= \deg_q \W_{(|\mu|-1,1),(|\nu|)}.
\ee

Moreover, if $\mu \neq (|\mu|)$ and $\mu \neq (|\mu|-1,1)$, we have $\mu_1 \leq |\mu|-2$ and  $-n_\mu \leq -2$,  hence
\be
\deg_q \W_{\mu,\nu} \leq (|\mu|-2)\cdot |\nu| - 2
- \half (|\mu|+|\nu|)
= \deg_q \W_{(|\mu|-2,2),(|\nu|)}.
\ee
\end{proof}

\subsection{Explicit formula for $\W_{(m),(n)}$}
In  last subsection
we have shown that
$$\deg_q \W_{(m),(n)} = mn-(m+n)/2.$$
In this subsection we will present an explicit formula for $\W_{(m),(n)}$.

\begin{lem}
The following identity holds:
\be \label{Wmn}
\W_{(m),(n)}(q)
= q^{mn-\frac{1}{2}(m+n)}\sum^{n}_{k=0} \frac{q^{-k(m+1)}}{[m]_{q^{-1}}![k]_{q^{-1}}!}.
\ee
\end{lem}
\begin{proof}
By \eqref{w}, \eqref{sh1} we have
\ben
\W_{(m),(n)}(q)= s_{(m)}(q^\varrho) s_{(n)}(q^{(m)+\varrho}).
\een
We now compute $s_{(m)}(q^\varrho)$ and $s_{(n)}(q^{(m)+\varrho})$.
Recall $s_{(m)} = h_m$,
hence
\be \label{eqn:sn(qrho)}
s_{(m)}(q^\varrho) = h_m(q^\rho) = \frac{q^{-m/2}}{[m]_{q^{-1}}!},
\ee
and
\ben
&& \sum_{n \geq 0} s_{(n)}(q^{(m)+\varrho})t^n
= \sum_{n \geq 0} h_{n}(q^{(m)+\varrho})t^n \\
& = & \frac{1}{1-q^{m-1/2}t} \cdot
\prod_{i=1}^\infty \frac{1}{1-q^{-i-1/2}t}
= \sum_{j=0}^\infty q^{mj-j/2} t^j \cdot
\sum_{j=0}^\infty \frac{q^{-k/2}}{[k]_{q^{-1}}!} t^k \\
& = & \sum_{n \geq 0} q^{-n/2} t^n \sum_{k=0}^n \frac{q^{m(n-k)-k}}{[k]_{q^{-1}}!}.
\een
Hence
\be \label{eqn:sn(qm+rho)}
s_{(n)}(q^{(m)+\varrho})
= q^{-n/2} \sum_{k=0}^n \frac{q^{mn-(m+1)k}}{[k]_{q^{-1}}!}.
\ee
This completes the proof.
\end{proof}

\section{Closed Formula for Leading Gopakumar-Vafa Invariants of Local $\PP^2$ Geometry}
\label{sec:P2}

In this section we will explicitly compute the Gopakumar-Vafa invariants $n^{(d-1)(d-2)/2-\delta}_d$
of $\kappa_{\P^2}$ for $0 \leq \delta \leq d-2$ based on duality with link invariants.
We will verify the predictions of Katz-Klemm-Vafa \cite{KKV} in this case.
This method will be generalized to the case of $0 \leq \delta \leq 2d-4$ where
no predictions of the closed formula have been made in the literature.
It will be applied to other toric Fano surfaces in Part II of this paper \cite{GZ}.

\subsection{A formula for the free energy $F^{\kappa_{\PP^2}}$}
Since $H_2(\PP^2;\Z) = \Z  H$,
where $H$ is the hyperplane class,
we will use an integer $d$ to index a homology class in $H_2(\PP^2;\Z)$
and write $K^g_d(\kappa_{\PP^2})$ for $K^g_{d H}(\kappa_{\PP^2})$.
The instanton part of the  free energy of $\kappa_{\PP^2}$ is
\be
F^{\kappa_{\PP^2}} = \sum_{g \geq 0} \lambda^{2g-2} \sum_{d > 0}
K^g_d(\kappa_{\PP^2}) \cdot t^d,
\ee
and the topological partition function of the local $\PP^2$ geometry
is defined by
$$Z^{\kappa_{\PP^2}} = \exp F^{\kappa_{\PP^2}}.$$
It can be explicitly computed \cite{AMV, I, zhou3} by  Chern-Simons invariants of Hopf link:
 \[
 Z^{\kappa_{\PP^2}}=\sum_{\mu^1,\mu^2,\mu^3}
 \W_{\mu^1,\mu^2}(q)\W_{\mu^2,\mu^3}(q)\W_{\mu^3,\mu^1}(q)\cdot
 q^{\frac{1}{2}\sum_{i=1}^3 \kappa_{\mu^i}}
 (-t)^{\sum_{i=1}^3 |\mu^i|}. \]
For a fixed $d\geq 1$,
denote the coefficient of $t^d$ in $Z^{\kappa_{\PP^2}}$ by $\I(d)$,
and denote by $F(d)$  the coefficient of $t^d$ in $F^{\kappa_{\PP^2}}$,
i.e., we have
\ben
&& Z^{\kappa_{\PP^2}} = 1+ \sum_{d \geq 1} \I(d) t^d, \\
&& F^{\kappa_{\PP^2}} =  \sum_{d \geq 1} F(d) t^d.
\een
So from $F^{\kappa_{\PP^2}} = \log Z^{\kappa_{\PP^2}}$ we have
\ben
&& F(1) = \I(1) \\
&& F(2) = \I(2)-\frac{1}{2}\I(1)^2, \\
&& F(3) = \I(3)-\I(1)I(2)+\frac{1}{3}\I(1)^3,
\een
in general
\be \label{eqn:F(d)inI(d)}
F(d) = \sum_{k \geq 1} \frac{(-1)^{k-1}}{k}
\sum_{d_1 + \cdots + d_k = d} \I(d_1) \cdots \I(d_k).
\ee
It is clear that
\be\label{Id}
\I(d)= (-1)^d \cdot \sum_{\sum_i |\mu^i|=d}
 q^{\frac{1}{2}(\sum_{i=1}^3 k_{\mu^i})}
\W_{\mu^1,\mu^2}(q)\W_{\mu^2,\mu^3}(q)\W_{\mu^3,\mu^1}(q),
\ee
so one has a way to explicitly compute all $F(d)$ and hence compute the Gopakumar-Vafa invariants
of $\PP^2$.
This has been used in \cite{AMV} to obtain a table of
Gopakumar-Vafa invariants of $\kappa_{\PP^2}$,
on which the observations in \cite{zhou1} were based.
In this section we will use it to obtain  closed formula
for some leading Gopakumar-Vafa invariants.

\subsection{Degree  estimates for $F^{\kappa_{\PP^2}}$}
Our starting point is the degree estimates of the terms in $\I(d)$
based on Lemma \ref{lm:NKappaEst} and Lemma \ref{lm:WEst}.
We regroup the terms in $\I(d)$ as follows:
\ben
\I(d)
& = & (-1)^d \cdot \sum_{\sum_i d_i=d}
q^{\frac{1}{2}\sum_i \kappa_{(d_i)}}
\W_{(d_1),(d_2)}(q)\W_{(d_2),(d_3)}(q)\W_{(d_3),(d_1)}(q) \\
& + & (-1)^d \cdot \sum_{\sum_i |\mu_i|=d}'
q^{\frac{1}{2}\sum_i \kappa_{\mu_i}}
\W_{\mu_1\mu_2}(q)\W_{\mu_2\mu_3}(q)\W_{\mu_3\mu_1}(q),
\een
where $\sum_{|\mu_1|+|\mu_2|+|\mu_3|=d}'$ means at least
one of $\mu^1, \mu^2, \mu^3$ is not of the form $(m)$.
By \eqref{Wmn},
\ben
\sum_{\sum_i d_i=d}   q^{\frac{1}{2}\sum_i \kappa_{(d_i)}} \cdot
\W_{(d_1),(d_2)}(q)\W_{(d_2),(d_3)}(q)\W_{(d_3),(d_1)}(q)
= q^{(d^2-3d)/2} W_d(q),
\een
where
\be
W_d(q)
=  \sum_{d_1+d_2+d_3=d}\sum^{d_1}_{k_1=0} \sum^{d_2}_{k_2=0}\sum^{d_3}_{k_3=0}\frac{q^{-k_2(d_1+1)}q^{-k_3(d_2+1)}q^{-k_1(d_3+1)}}
{\prod_{i=1}^3 ([d_i]_{q^{-1}}![k_i]_{q^{-1}}!)}.
\ee
In particular,
we have
\be
\deg_q \sum_{\sum_i d_i=d}
q^{\frac{1}{2}\sum_i \kappa_{(d_i)}}
\W_{(d_1),(d_2)}(q)\W_{(d_2),(d_3)}(q)\W_{(d_3),(d_1)}(q)
= (d^2-3d)/2.
\ee

\begin{lem} \label{degreeofId}
Assume that $\mu^1, \mu^2, \mu^3$ such that $|\mu^i| = d_i$, $i=1,2,3$,
and at least one of them is not of the form $(m)$.
\ben
&& \deg_q (q^{\frac{1}{2}\sum_i \kappa_{\mu_i}}
\W_{\mu_1\mu_2}(q)\W_{\mu_2\mu_3}(q)\W_{\mu_3\mu_1}(q)) \\
&  - & \deg_q( q^{\frac{1}{2}\sum_{i=1}^3 \kappa_{(d_i)}}
\cdot \W_{(d_1),(d_2)}(q)\W_{(d_2),(d_3)}(q)\W_{(d_3),(d_1)}(q)) \\
& \leq & - (d_1+d_2+d_3)-2.
\een
\end{lem}

\begin{proof}
There are three cases to consider: Case 1. Exactly one of $\mu^i$'s is not of the form $(m)$;
Case 2.  Exactly two of $\mu^i$'s are not of the form $(m)$;
Case 3.  Exactly three of $\mu^i$'s are not of the form $(m)$.
For Case 1, assume $\mu^1=(d_1)$, $\mu^2 = (d_2)$, $|\mu^3|=d_3$
but $\mu^3 \neq (d_3)$,
then by Lemma \ref{lm:NKappaEst} we have
\ben
\deg_q q^{\kappa_{\mu^3}/2} -\deg_q q^{\kappa_{(d_3)}/2} \leq  - d_3,
\een
and by Lemma \ref{lm:WEst} we have
\ben
&& \deg_q \W_{(d_2), \mu^3}(q) - \deg_q \W_{(d_2), (d_3)}(q) \leq - (d_2+1), \\
&& \deg_q \W_{\mu^3, (d_1)}(q) - \deg_q \W_{(d_3), (d_1)}(q) \leq - (d_1+1),
\een
therefore,
\ben
&& \deg_q( q^{\frac{1}{2}\sum_{i=1}^2 \kappa_{(d_i)}+\half \kappa_{\mu^3}}
\cdot \W_{(d_1),(d_2)}(q)\W_{(d_2),\mu^3}(q)\W_{\mu^3,(d_1)}(q)) \\
& - & \deg_q( q^{\frac{1}{2}\sum_{i=1}^3 \kappa_{(d_i)}}
\cdot \W_{(d_1),(d_2)}(q)\W_{(d_2),(d_3)}(q)\W_{(d_3),(d_1)}(q)) \\
& \leq & - d_3 - (d_2+1) - (d_1+1) \\
& = & - d - 2.
\een
The other two cases can be treated in the same fashion.
\end{proof}

We will use the following notations:
For a series $f(q)$ and $g(q)$ of the form
\ben
&& f(q) = a_n q^n + a_{n-1} q^{n-1} + a_{n-2} q^{n-2} + \cdots
\een
where $n \in \Z$,
and $m \in \Z$ such that $m  \leq n$,
define $f(q)|_{g^{\geq m}}$ by
\ben
 f(q) =  a_n q^n + a_{n-1} q^{n-1} + a_{n-2} q^{n-2} + \cdots + a_m q^m.
\een
If $f(q)$ and $g(q)$ are series of the above form,
and
$f(q)|_{q^{\geq m}} = g(1)|_{g^{\geq m}}$,
then we write
$$f(q) = g(q)|_{q^{\geq m}}.$$

So we have proved the following:

\begin{lem}\label{Idwd}
For $d \geq 1$,
\be \label{I(d)}
\I(d) = (-1)^d q^{(d^2-3d)/2} W_d(q)
 \biggr|_{q^{\geq (d^2-3d)/2-d-1}},
\ee
\end{lem}

\begin{prop}
For $d \geq 1$,
$\deg_q F(d) \leq (d^2-3d)/2$.
Furthermore,
all the terms of the form $q^{(d^2-3d)/2 - \delta}$ for $0 \leq \delta \leq d-2$
come from $\I(d)$.
I.e.,
\be \label{F(d)=I(d)}
F(d) = \I(d)|_{q^{\geq (d^2-3d)/2-d+2}}.
\ee
\end{prop}

\begin{proof}
Suppose that $k>1$,
and $d_1, \dots, d_k > 0$ are positive integers such that $d_1+\cdots + d_k = d$,
then
\ben
&& \deg_q(\I(d_1) \cdots \I(d_k))
= \half [(d_1^2-3d_1) + \cdots + (d_k^2-3d_k)] \\
& = &  \half(d^2-3d) - \sum_{1 \leq i < j \leq k} d_id_j
\leq \half(d^2-3d) - (d-1).
\een
Here we have use the following inequality:
For positive integers $d_1, \dots, d_k$ ($k>1$),
one has
\be\label{didj}
\sum_{1 \leq i < j \leq k} d_id_j \geq \sum_{i=1}^k d_i - 1.
\ee
This can be easily proved by induction.
The proof is completed by noting \eqref{eqn:F(d)inI(d)}.
\end{proof}

\subsection{A vanishing result for Gopakumar-Vafa invariants for local $\PP^2$}
For the local $\PP^2$ geometry
the Gopakumar-Vafa Conjecture has been proved by Peng \cite{P} and Konishi \cite{Ko1, Ko2}.
So we have
\be
F=\sum_{d\geq 1}\sum_{g \geq 0} n^g_d (\kappa_{\PP^2}) \sum_{k>0} \frac{(-1)^{g-1}}{k}(q^{\frac{k}{2}}-q^{-\frac{k}{2}})^{2g-2}t^{kd},
\ee
for some integers $n^g_d(\kappa_{\PP^2})$.
Therefore
by comparing the coefficients of $t^d$,
one gets:
\be \label{F(d)}
F(d)= \sum_{k|d}\sum_{g \geq 0} n^g_{d/k} (\kappa_{\PP^2})  \frac{(-1)^{g-1}}{k}(q^{\frac{k}{2}}-q^{-\frac{k}{2}})^{2g-2}.
\ee
This must be a finite sum because we have
\be
\deg_q F(d) = (d^2-3d)/2.
\ee
It follows that if $n^g_{d/k} \neq 0$,
then
\be
k\cdot (g-1) \leq (d^2-3d)/2.
\ee
In particular,
taking $k=1$
we get

\begin{prop}
For the local $\PP^2$ geometry,
if $n^g_d \neq 0$,
then
\be
g \leq g(d): = (d-1)(d-2)/2.
\ee
\end{prop}

This was empirically observed in \cite{GV, AMV}
and proved by Peng \cite{P} by similar arguments.

\subsection{Leading Gopakumar-Vafa invariants for local $\PP^2$}
With the above vanishing result one can rewrite \eqref{F(d)} as follows:
\be
F(d)= \sum_{k|d}\sum_{g = 0}^{g(d/k)} n^g_{d/k} (\kappa_{\PP^2})
 \frac{(-1)^{g-1}}{k}(q^{\frac{k}{2}}-q^{-\frac{k}{2}})^{2g-2}.
\ee
We can write it as
\be
F(d)= \sum_{k|d}F^k(d),
\ee
where
\be
F^k(d)= \sum_{g = 0}^{g(d/k)} n^g_{d/k} (\kappa_{\PP^2})
 \frac{(-1)^{g-1}}{k}(q^{\frac{k}{2}}-q^{-\frac{k}{2}})^{2g-2}.
\ee
Because for $k > 1$ and $0 \leq g \leq g(d/k)$, we have
\ben
&& \deg_q (q^{\frac{k}{2}} - q^{\frac{k}{2}})^{2g-2}
=  \frac{k}{2} \cdot (2g-2) \\
& \leq & \frac{k}{2}\cdot(2g(d/k)-2)=\frac{1}{2}(\frac{d^2}{k}-3d)\\
& = & \half (d^2-3d+2) - [\frac{d}{2}(d- \frac{d}{k}) +1]
\leq g(d) - d,
\een
moreover if $d>3$ we have
$$
\deg_q (q^{\frac{k}{2}} - q^{\frac{k}{2}})^{2g-2}\leq
g(d) - (\frac{d^2}{4}+1) \leq g(d)-(2d-3)
.
$$
So one gets,
\be \label{F(d)=N}
F(d)
= F^1(d)= \sum_{g = 0}^{g(d)} n^g_d (\kappa_{\PP^2}) {(-1)^{g-1}}(q^{\frac{1}{2}}-q^{-\frac{1}{2}})^{2g-2}
\biggr|_{q^{\geq g(d)-d+1}},
\ee
and if $d>3$
\be \label{F(d)=N2}
F(d)
= F^1(d)
\biggr|_{q^{\geq g(d)-2d+4}}.
\ee
In other words,
\be
\sum_{g = g(d)-d+1}^{g(d)} n^g_d (\kappa_{\PP^2}) {(-1)^{g-1}}(q^{\frac{1}{2}}-q^{-\frac{1}{2}})^{2g-2}
= F(d)|_{q^{\geq g(d)-d+1}}
\ee
Now we combine this equation with \eqref{F(d)=I(d)} and \eqref{I(d)} to get:
\be \label{eqn:nW}
\begin{split}
& \sum_{g = g(d)-d+2}^{g(d)} n^g_d (\kappa_{\PP^2}) {(-1)^{g-1}}(q^{\frac{1}{2}}-q^{-\frac{1}{2}})^{2g-2} \\
= & (-1)^d \cdot q^{(d^2-3d)/2}W_d(q)|_{q^{\geq g(d)-d+1}}.
\end{split}
\ee

The following easy result will be useful:

\begin{lem} Suppose that $f(q) = g(q)|_{q^{\geq m}}$.
Then

(a)  $qf(q) = qg(q)|_{q^{\geq m+1}}$.

(b) For $j >  0$, $(1-q^{-j}) f(q) = (1-q^{-j}) g(q)|_{q^{\geq m}}$.
\end{lem}

\begin{proof}
Suppose that
$f(q) = a_n q^n + a_{n-1} q^{n-1} + \cdots + a_m q^m + a_{m-1} q^{m-1} + \cdots$,
and $g(q) = a_n q^n + a_{n-1} q^{n-1} + \cdots + a_m q^m + b_{m-1} q^{m-1} + \cdots$,
then we have
\ben
&& q f(q) =  a_n q^{n+1} + a_{n-1} q^{n} + \cdots + a_m q^{m+1}
+ a_{m-1} q^{m} + \cdots, \\
&& qg(q) = a_n q^{n+1} + a_{n-1} q^{n} + \cdots + a_m q^{m+1} + b_{m} q^{m-1} + \cdots,
\een
this proves (a).
Similarly,
\ben
&& (1-q^{-j})f(q) =  a_n q^{n} + a_{n-1} q^{n-1} + \cdots + a_m q^{m}
+ a_{m-1} q^{m-1} + \cdots \\
&& \;\;\;\;\;\;\;\;\; - a_n q^{n-j} + a_{n-1} q^{n-1-j} + \cdots + a_m q^{m-j}
+ a_{m-1} q^{m-1-j} + \cdots \\
&& (1-q^{-j})g(q) =  a_n q^{n} + a_{n-1} q^{n-1} + \cdots + a_m q^{m}
+ b_{m-1} q^{m-1} + \cdots \\
&& \;\;\;\;\;\;\;\;\; - a_n q^{n-j} + a_{n-1} q^{n-1-j} + \cdots + a_m q^{m-j}
+ b_{m-1} q^{m-1-j} + \cdots,
\een
hence (b) is evident.
\end{proof}

Multiplying both sides of \eqref{eqn:nW} by $q(1-q^{-1})^2$,
we get by the above Lemma:
\be
\begin{split}
& \sum_{g = g(d)-d+2}^{g(d)} n^g_d (\kappa_{\PP^2}) {(-1)^{g-1}}(q^{\frac{1}{2}}-q^{-\frac{1}{2}})^{2g} \\
= & (-1)^d \cdot q^{g(d)} (1-q^{-1})^2 \cdot W_d(q)|_{q^{\geq g(d)-d+2}}.
\end{split}
\ee
It follows that for fixed $d$,
when $g(d)-d+2 \leq g \leq g(d)$,
$n^g_d(\kappa_{\PP^2})$ is determined by the coefficients of $q^l$ in $W_d(q)$ for $-d+2 \leq l \leq 0$.
Such invariants will be referred as the leading Gopakumar-Vafa invariants.

We will not directly evaluate $n^g_d(\kappa_{\PP^2})$,
but instead will consider the transformed Gopakumar-Vafa invariants.

\subsection{Transformed Gopakumar-Vafa invariants}
Recall the transformed Gopakumar-Vafa invariants $N_d^g(\kappa_{\PP^2})$ is defined by:
$$\sum_{g=0}^{g(d)} N_d^g (\kappa_{\PP^2}) (q^g+q^{g-2}+\cdots+q^{-g})
= \sum_{g = 0}^{g(d)} (-1)^{g}n_d^g (\kappa_{\PP^2})(q^{\frac{1}{2}}-q^{-\frac{1}{2}})^{2g},
$$
Multiply both sides by $1-q^{-2}$ to get:
\ben
&& \sum_{g=0}^{g(d)} N_d^g (\kappa_{\PP^2}) \cdot (q^g-q^{-g-2})
= (1-q^{-2})\sum_{g = 0}^{g(d)} (-1)^{g}n_d^g (\kappa_{\PP^2})(q^{\frac{1}{2}}-q^{-\frac{1}{2}})^{2g}.
\een
We take the terms of the form $q^k$ for $k \geq g(d) - d+2$ to get
\be \label{eqn:NinW}
\begin{split}
& \sum_{g=g(d)-d+2}^{g(d)} N_d^g (\kappa_{\PP^2})  q^g \\
= & (1-q^{-2}) \sum_{g = 0}^{g(d)} (-1)^{g}n_d^g (\kappa_{\PP^2})(q^{\frac{1}{2}}-q^{-\frac{1}{2}})^{2g}
\biggr|_{q^{\geq g(d)-d+2}} \\
= & (-1)^{d-1} (1-q^{-2})(1-q^{-1})^2 q^{g(d)} W_d(q)|_{q^{\geq g(d)-d+2}}.
\end{split}
\ee
So in order to find the leading transformed Gopakumar-Vafa invariants
$N^g_d(\kappa_{\PP^2})$ for $g(d)-d+2 \leq g \leq g(d)$,
we need to find the coefficients of $q^{-\delta}$ in $W_d(q)$
for $\delta \leq d-2$.

\subsection{Calculations of $W_d(q)$ by hypergeometric series}
In this subsection we will evaluate  the following summation:
\be
W_d(q)
=  \sum_{d_1+d_2+d_3=d}\sum^{d_1}_{k_1=0} \sum^{d_2}_{k_2=0}\sum^{d_3}_{k_3=0}\frac{q^{-k_2(d_1+1)}q^{-k_3(d_2+1)}q^{-k_1(d_3+1)}}
{\prod_{i=1}^3 ([d_i]_{q^{-1}}![k_i]_{q^{-1}}!)}. \label{wd}
\ee

For $m \geq 0$,
we introduce an operator $T_m^x: \C[[x]] \to \C$ as follows:
For $f(x) = a_0 + a_1 x + a_2x^2+ \cdots$,
define
\be
T_m^x f(x): = a_0 + a_1  + a_2 + \cdots + a_m.
\ee

\begin{lem}
The following identities hold:

\begin{align}
&\sum_{k=1}^m \frac{q^{-k(d+1)}}{[k]_{q^{-1}}!}
= T^x_m \biggl(\frac{1}{( xq^{-d-1};q^{-1})_\infty}-1 \biggr), \label{msum1} \\
& \frac{1}{[d_2]_{q^{-1}}!}\sum_{k=0}^{d_3} \frac{q^{-k(d_2+1)}}{[k]_{q^{-1}}!}
=  \frac{1}{[\infty]_{q^{-1}}!}
T^x_{d_3}\biggl(\frac{(q^{-d_2-1};q^{-1})_{\infty}}{(x q^{-d_2-1};q^{-1})_{\infty}}\biggr)
.\label{msum2}
\end{align}
\end{lem}

\begin{proof}
Take $z = xq^{-d-1}$ in \eqref{exp},
\be
\sum_{n\geq 0} \frac{x^n q^{-n(d+1)}}{[n]_q!}
=\frac{1}{(xq^{-d-1},q)_{\infty}}.
\ee
Then one applies $T_m^x$ on both sides to get the first identity.
For the second identity,
just note
\be
 \frac{1}{[d_2]_{q^{-1}}!}=  \frac{(q^{-d_2-1};q^{-1})_{\infty}}{[\infty]_{q^{-1}}!}.
\ee
\end{proof}

By \eqref{msum2},
\ben
W_d(q)
=  \sum_{\sum_i {d_i}=d} \left(\frac{1}{[\infty]_{q^{-1}}!}
T^x_{d_3}\biggl(\frac{(q^{-d_2-1};q^{-1})_{\infty}}{(x q^{-d_2-1};q^{-1})_{\infty}}-1\biggr)+\frac{1}{[\infty]_{q^{-1}}!} \right)\\
\cdot \left(\frac{1}{[\infty]_{q^{-1}}!}
T^x_{d_2}\biggl(\frac{(q^{-d_1-1};q^{-1})_{\infty}}{(x q^{-d_1-1};q^{-1})_{\infty}}-1\biggr)+\frac{1}{[\infty]_{q^{-1}}!}\right)\\
\cdot\left(\frac{1}{[\infty]_{q^{-1}}!}
T^x_{d_1}\biggl(\frac{(q^{-d_3-1};q^{-1})_{\infty}}{(x q^{-d_3-1};q^{-1})_{\infty}}-1\biggr)+\frac{1}{[\infty]_{q^{-1}}!}\right)
\een

Now we can regroup the terms on the right-hand side of \eqref{wd} as follows:

\be
W_d(q) = W_d^1(q) + W^2_d(q) + W_d^3(q)+W^4_d(q),
\ee
where
\bea
&& W_d^1(q)= \sum_{\sum d_i=d} \frac{1}{([\infty]_{q^{-1}}!)^3}
=\frac{(d+1)(d+2)}{2\cdot ([\infty]_{q^{-1}}!)^3}, \label{eqn:W1} \\
&& W_d^2(q)= 3 \sum_{\sum d_i=d} \frac{1}{([\infty]_{q^{-1}}!)^3}  \cdot T^x_{d_2}\left(\frac{(q^{-d_1-1};q^{-1})_{\infty}}{(xq^{-d_1-1};q^{-1})_{\infty}}-1\right),
\\
&& W_d^3(q)
= 3 \sum_{\sum d_i=d} \frac{1}{([\infty]_{q^{-1}}!)^3} \cdot T^x_{d_2}
\biggl(\frac{(q^{-d_1-1};q^{-1})_{\infty}}{(xq^{-d_1-1};q^{-1})_{\infty}}-1\biggr)\\
&&\qquad \qquad
\cdot T^x_{d_3}
\biggl(\frac{(q^{-d_2-1};q^{-1})_{\infty}}{(xq^{-d_2-1};q^{-1})_{\infty}}-1\biggr), \nonumber\\
&& W_d^4(q)
= 3 \sum_{\sum d_i=d} \frac{1}{([\infty]_{q^{-1}}!)^3} \cdot T^x_{d_2}
\biggl(\frac{(q^{-d_1-1};q^{-1})_{\infty}}{(xq^{-d_1-1};q^{-1})_{\infty}}-1\biggr)\\
&&\qquad \quad
\cdot T^x_{d_3}
\biggl(\frac{(q^{-d_2-1};q^{-1})_{\infty}}{(xq^{-d_2-1};q^{-1})_{\infty}}-1\biggr)
\cdot T^x_{d_1}
\biggl(\frac{(q^{-d_3-1};q^{-1})_{\infty}}{(xq^{-d_3-1};q^{-1})_{\infty}}-1\biggr). \nonumber
\eea

\begin{lem} \label{lm:leading}
We have the following formulas for the leading term:
\bea
&& \frac{1}{[\infty]_{q^{-1}}!}-\frac{1}{[d]_{q^{-1}}!}=q^{-(d+1)}\cdot(1+a_1q^{-1}+a_2q^{-2}+\cdots), \\
&& T^x_m \biggl(\frac{(q^{-d-1};q^{-1})_{\infty}}{(x q^{-d-1};q^{-1})_{\infty}}-1 \biggr)
=q^{-(d+1)(m+1)}\cdot(b_0+b_1q^{-1} +\cdots).
\eea
\end{lem}

\begin{proof}
These follow from straightforward calculations:
\ben
&& \frac{1}{[\infty]_{q^{-1}}!}-\frac{1}{[d]_{q^{-1}}!}
= \frac{1}{[d]_{q^{-1}}!}\biggl( \prod_{j=d+1}^\infty \frac{1}{1-q^{-j}} - 1\biggr) \\
& = & \frac{1}{[d]_{q^{-1}}!} \cdot (q^{-(d+1)} + \cdots)
= q^{-(d+1)} + \cdot.
\een
Similarly,
\ben
&& T^x_m\biggl(\frac{(q^{-d-1};q^{-1})_{\infty}}{(x q^{-d-1};q^{-1})_{\infty}}-1\biggr) \\
& = & T^x_m \biggl( \frac{1}{(x q^{-d-1};q^{-1})_{\infty}}
\cdot \frac{1}{\frac{1}{(q^{-d-1};q^{-1})_{\infty}}}-1 \biggr)  \\
& = & \frac{\sum_{k=0}^m \frac{q^{-k(d+1)}}{[k]_{q^{-1}}!} }
{\sum_{k=0}^\infty \frac{q^{-k(d+1)}}{[k]_{q^{-1}}!}} - 1
= \frac{1}{1 + \frac{\sum_{k=m+1}^\infty \frac{q^{-k(d+1)}}{[k]_{q^{-1}}!}}
{\sum_{k=0}^m \frac{q^{-k(d+1)}}{[k]_{q^{-1}}!} }}  - 1 \\
& = & - q^{-(m+1)(d+1)} + \cdots,
\een
where in the third identity we have used \eqref{exp}.
\end{proof}

Using Lemma \ref{lm:leading} we get
\ben
&& W_d^3(q)
=  3 \sum_{\sum d_i=d}  (q^{-(d_1+1)(d_2+1)}+ \cdots )
\cdot (-q^{-(d_2+1)(d_3+1)} + \cdots ),
\een
it follows that $\deg_q W_d^{c}(q) \leq - (d+2)$.
Similarly,
\ben
&&\deg_q W_d^4(q) = \max\{- (d_1+1)(d_2+1)-(d_2+1)(d_3+1)
\\
&& \qquad\qquad  -(d_3+1)(d_1+1): d_1+d_2+d_3=d\}
\leq  -(2d+3).
\een
Since we are concerned with the coefficients of $q^l$ for $-d+2 \leq l \leq 0$,
such part of contributions are all from $W^a_d(q)$ and $W^b_d(q)$.

\subsection{Computations of $W^2_d(q)$}

\begin{prop} \label{E1}
We have
\ben
&&\sum_{d\geq 0}  \left(\frac{1}{[\infty]_{q^{-1}}!}-\frac{1}{[d]_{q^{-1}}!} \right)
= \frac{1}{[\infty]_{q^{-1}}!} \sum_{i\geq 1} \frac{q^{-i}}{1-q^{-i}}.
\een
\end{prop}
\begin{proof}
Consider
\ben
\sum_{d\geq 0} t^d \left(\frac{1}{[d]_{q^{-1}}!}-\frac{1}{[\infty]_{q^{-1}}!} \right)
&=& \frac{1}{(t,q^{-1})_{\infty}} - \frac{1}{(1-t)(q^{-1},q^{-1})_{\infty}}\\
&=&
\frac{1}{t-1}(\frac{1}{(tq^{-1},q^{-1})_{\infty}} - \frac{1}{(q^{-1},q^{-1})_{\infty}}).
\een
Hence
\ben
 \lim_{t\to 1} \sum_{d\geq 0} t^d \left(\frac{1}{[\infty]_{q^{-1}}!}-\frac{1}{[d]_{q^{-1}}!} \right)
 &=& \frac{d}{dt} \frac{1}{(tq^{-1},q^{-1})_{\infty}}  \biggr|_{t=1} \\
 &=& \frac{1}{[\infty]_{q^{-1}}!} \sum_{i\geq 1} \frac{q^{-i}}{1-q^{-i}}.
\een
\end{proof}
\begin{prop}\label{E2}
Write
\ben
W_\infty^2 = \frac{3}{[\infty]_{q^{-1}}!^3} \sum_{d_1, d_2 \geq 0}
T^x_{d_2} \left(\frac{(q^{-d_1-1};q^{-1})_{\infty}}{(xq^{-d_1-1};q^{-1})_{\infty}}-1\right),
\een
then we have:
\be \label{eqn:W2inf}
W_\infty^2 = - \frac{3}{[\infty]_{q^{-1}}!^3} \sum_{i\geq 1} \frac{q^{-i}}{(1-q^{-i})^2}.
\ee
\end{prop}

\begin{proof}We use Abel summation to carry out $\sum_{d_1 \geq 0}$.
\ben
&& \sum_{d_1 \geq 0}
\left(\frac{(q^{-d_1-1};q^{-1})_{\infty}}{(xq^{-d_1-1};q^{-1})_{\infty}}-1\right) \\
& = & \lim_{t\to 1} \sum_{d_1 \geq 0}
\left(\frac{(q^{-d_1-1};q^{-1})_{\infty}}{(xq^{-d_1-1};q^{-1})_{\infty}} t^{d_1}-t^{d_1}\right) \\
& = & \lim_{t\to 1} \biggl(\frac{(q^{-1};q^{-1})_{\infty}}{(xq^{-1};q^{-1})_{\infty}}
\sum_{d_1 \geq 0} \frac{(xq^{-1};q^{-1})_{d_1}}{(q^{-1};q^{-1})_{d_1}} t^{d_1}
- \frac{1}{1-t} \biggr) \\
& = &  \lim_{t\to 1} \biggl(\frac{(q^{-1};q^{-1})_{\infty}}{(xq^{-1};q^{-1})_{\infty}}
\cdot \frac{(xq^{-1}t;q^{-1})_{\infty}}{(t;q^{-1})_{\infty}}
- \frac{1}{1-t} \biggr).
\een
In the last equality we have used the $q$-binomial identity \eqref{binomial}.
Note
$$(t;q^{-1})_\infty = (1-t) \cdot (tq^{-1};q^{-1})_\infty.$$
So we can move forward as follows:
\ben
&& \sum_{d_1 \geq 0}
\left(\frac{(q^{-d_1-1};q^{-1})_{\infty}}{(xq^{-d_1-1};q^{-1})_{\infty}}-1\right) \\
& = & - \lim_{t\to 1} \frac{1}{t-1} \biggl(\frac{(q^{-1};q^{-1})_{\infty}}{(xq^{-1};q^{-1})_{\infty}}
\cdot \frac{(xq^{-1}t;q^{-1})_{\infty}}{(tq^{-1} ;q^{-1})_{\infty}}
- 1 \biggr) \\
& = & - \frac{d}{dt} \biggl(\frac{(q^{-1};q^{-1})_{\infty}}{(xq^{-1};q^{-1})_{\infty}}
\cdot \frac{(xq^{-1}t;q^{-1})_{\infty}}{(tq^{-1} ;q^{-1})_{\infty}}\biggr) \biggr|_{t=1} \\
& = &\sum_{i\geq 1}\left(\frac{x q^{-i}}{1- x  q^{-i}}-\frac{q^{-i}}{1-q^{-i}}\right).
\een
Note
\ben
&& T^x_{d_2} \biggl(\frac{xq^{-i}}{1- x  q^{-i}} \biggr)
= \sum_{j=1}^{d_2} x^j q^{-ij}|_{x=1} =   \sum_{j=1}^{d_2} q^{-ij}
= \frac{(1-q^{-i d_2})q^{-i}}{1-   q^{-i}},
\een
so we have
\ben
&  & \sum_{d_2 \geq 0}T^x_{d_2}  \sum_{d_1 \geq 0}
\left(\frac{(q^{-d_1-1};q^{-1})_{\infty}}{(xq^{-d_1-1};q^{-1})_{\infty}}-1\right) \\
& = & \sum_{d_2 \geq 0}T^x_{d_2} \sum_{i\geq 1}\left(\frac{x q^{-i}}{1- x  q^{-i}}-\frac{q^{-i}}{1-q^{-i}}\right) \\
& =& \sum_{d_2 \geq 0} \sum_{i\geq 1}\left(\frac{(1-q^{-i d_2})q^{-i}}{1-   q^{-i}}-\frac{q^{-i}}{1-q^{-i}}\right) \\
& = & - \sum_{d_2 \geq 0} \sum_{i\geq 1} \frac{(q^{-i})^{ (d_2+1)}}{1-   q^{-i}}
= - \sum_{i\geq 1} \frac{q^{-i}}{(1-   q^{-i})^2}.
\een
\end{proof}

As a corollary we have

\begin{prop} We have
\be
W^2_d(q)
= -\frac{3}{[\infty]_{q^{-1}}!^3} \sum_{i\geq 1} \frac{q^{-i}}{(1-   q^{-i})^2}\biggl|_{q^{\geq -d}}.
\ee
\end{prop}
\begin{proof}
Just consider
$$
W^{2'}_d(q) = W^2_d(q) - W^2_\infty(q) = \frac{-3}{[\infty]_{q^{-1}}!^3} \sum_{d_1+ d_2 >d}
T^x_{d_2} \left(\frac{(q^{-d_1-1};q^{-1})_{\infty}}{(xq^{-d_1-1};q^{-1})_{\infty}}-1\right),
$$
by Lemma \ref{lm:leading} we have
$$
\deg_q W^{2'}_d(q) = \max\{-(d_1+1)(d_2+1) :  d_1+d_2 >d\} \leq -(d+1).
$$
\end{proof}
Combining all the results in this subsection,
we get

\begin{thm} \label{thm:Wd}
When $W_d(q)$ is expanded into a series in $\C[q,q^{-1}]]$,
one has
\be
W_d(q) = \frac{1}{[\infty]_{q^{-1}}!^3}\biggl(\binom{d+2}{2}
- 3 \sum_{i\geq 1} \frac{q^{-i}}{(1-   q^{-i})^2}\biggr)\biggl|_{q^{\geq -d}}.
\ee
\end{thm}

\subsection{Computation of leading transformed GV invariants}

By \eqref{eqn:NinW} we have
\be
\sum_{\delta=0}^{d-2} (-1)^{d-1}  N_d^{g(d)-\delta} (\kappa_{\PP^2})  q^{-\delta}
= (1-q^{-2})(1-q^{-1})^2  W_d(q)|_{q^{\geq -d+2}}.
\ee
Hence by Theorem \ref{thm:Wd} we get:
\be
\begin{split}
& \sum_{\delta=0}^{d-2} (-1)^{d-1}  N_d^{g(d)-\delta} (\kappa_{\PP^2})  q^{-\delta} \\
= &  \frac{[1]_{q^{-1}}^2[2]_{q^{-1}}}{[\infty]_{q^{-1}}!^3}\biggl(\binom{d+2}{2}
- 3 \sum_{i\geq 1} \frac{q^{-i}}{(1-   q^{-i})^2}\biggr)\biggl|_{q^{\geq -d+2}}.
\end{split}
\ee
In other words,
we have proved the following

\begin{thm} \label{thm:Mdelta}
For $\delta \geq 0$,
when $d \geq \delta +2$,
$M^\delta_d(\kappa_{\PP^2}):=(-1)^{d-1}  N_d^{g(d)-\delta} (\kappa_{\PP^2})$
is a quadratic polynomial $M_\delta(d)$ in $d$,
and they have the following generating series
\be
\sum_{\delta \geq 0} M_{\delta}(x) q^\delta
= \frac{[1]^2[2]}{[\infty]!^3}\biggl(\binom{x+2}{2}
- 3 \sum_{i\geq 1} \frac{q^{i}}{(1-   q^{i})^2}\biggr).
\ee
\end{thm}

The first few terms as in the Introduction where we use $M^\delta_d$ for $M_{\delta}(d)$ are
\bea
&& M_0(d) = \frac{1}{2}(d^2+3d+2), \\
&& M_1(d) = \frac{1}{2} (d^2+3d-4), \;\; (d \geq 3) \\
&& M_2(d) = \frac{3}{2} (d^2 +3d-6), \;\;\; (d \geq 4) \\
&& M_3(d) = 3(d^2+3d)-24, \;\;\; (d \geq 5) \\
&& M_4(d) = 6 (d^2+3d-11), \;\; (d \geq 6) \\
&& M_5(d) = \frac{21}{2}(d^2+3d)-144, \;\;\; (d \geq 7) \\
&& M_6(d) = 20(d^2+3d-16), \;\;\; (d \geq 8) \\
&& M_7(d) = \frac{67}{2}(d^2+3d)-626, \;\;\; (d \geq 9) \\
&& M_8(d) = \frac{117}{2}(d^2+3d)-1233, \;\;\; (d \geq 10).
\eea

\subsection{Closed formula for leading GV invariants of local $\PP^2$ geometry}

Recall the GV invariants $n^g_d$ and the transformed GV invariants $N^g_d$
are related as follows:
\be
\sum_{g=0}^{g(d)} (-1)^g n^g_d (q^{1/2}-q^{-1/2})^{2g}
= \sum_{g=0}^{g(d)} N^g_d (q^g+q^{g-2} + \cdots q^{-g}).
\ee
Denote by $SP(q)$ the space
of Laurent polynomials $p(q) \in \C[q,q^{-1}]$ such that
$$p(q^{-1}) = p(q).$$
Consider the following two bases of $SP(q)$:
\bea
&& S_g = (q^{1/2}-q^{-1/2})^{2g}, \label{eqn:Sg} \\
&& R_g = (q^g+q^{g-2} + \cdots + q^{-g}).
\eea

\begin{lem} \label{lm:SgRg}
The bases $\{S_g\}$ and $\{R_g\}$ are related as follows:
\bea
&& S_g = \sum_{j=0}^g (-1)^{g-j} \biggl(\binom{2g}{g-j} -  \binom{2g}{g-j-2} \biggr)R_{j}, \label{eqn:SgInRg} \\
&& R_g = \sum_{j=0}^g \binom{g+j+1}{g-j} S_j. \label{eqn:RgInSg}
\eea
\end{lem}

\begin{proof}
By applying the binomial expansion
to \eqref{eqn:Sg} one gets:
\be
S_g =  g^g - \binom{2g}{1} q^{g-1} + \binom{2g}{2} q^{g-2}
+ \cdots + q^{-g}.
\ee
From this one gets the first identity.
To prove the second identity,
consider the generating series:
\ben
&& S(t) = \sum_{g=0}^\infty S_gt^g = \frac{1}{1-(q^{1/2}-q^{-1/2})^2t}, \\
&& R(t) = \sum_{g=0}^\infty R_gt^g
= \frac{1}{q-q^{-1}} \biggl(\frac{q}{1-qt} - \frac{q^{-1}}{1-q^{-1}t}\biggr)
= \frac{1}{(1-qt)(1-q^{-1}t)}.
\een
By an elementary calculation one can see that
\ben
R(t) = \frac{1}{(1-t)^2} S\biggl(\frac{t}{(1-t)^2}\biggr).
\een
I.e.,
\ben
\sum_{g\geq 0} R_g t^g & = & \sum_{j \geq 0} S_j \cdot \frac{t^j}{(1-t)^{2j+2}} \\
& = & \sum_{g \geq 0} t^g \sum_{j=0}^g \binom{g+j+1}{g-j} S_j.
\een
\end{proof}

As a consequence we have

\begin{lem}
The GV numbers $\{n^g_d\}$ and $\{N^g_d\}$ are related as follows:
\bea
&& N^h_d = (-1)^h \sum_{g \geq h}  \biggl(\binom{2g}{g-h} -  \binom{2g}{g-h-2} \biggr), \\
&& n^g_d = (-1)^g \sum_{h \geq g} N^h_d \binom{g+h+1}{h-g}.
\eea
\end{lem}

Applying this result,
we have for the local $\PP^2$ geometry,
\be
n^{g(d)-\delta}_d(\kappa_{\PP^2})
= (-1)^{g(d)-\delta}\sum_{j=0}^\delta
N^{g(d)-j}_d \cdot \binom{2g(d)-\delta-j+1}{\delta-j}.
\ee
Recall $g(d) = \frac{(d-1)(d-2)}{2}$ and by Theorem \ref{thm:Mdelta},
when $d \geq \delta+2$,
the number $(-1)^{d-1}N^{g(d)-j}_d$ is a quadratic polynomial,
it follows that under the same condition,
$(-1)^{g(d)+d-1-\delta}n^{g(d)-\delta}_d(\kappa_{\PP^2})$ is a polynomial $n_\delta(d)$
given by:
\be
n_\delta(d): = \sum_{j=0}^\delta
M_j(d) \cdot \binom{d^2-3d+3-\delta-j}{\delta-j}.
\ee
The first few terms are
\bea
&& n_0(d) = \frac{1}{2}(d^2+3d+2), \\
&& n_1(d) = \frac{1}{2} d(d-1)(d^2+d-3), \;\; (d \geq 3) \\
&& n_2(d) = \frac{1}{4} (d-1)(d^5-2d^4-6d^3+9d^2+36), \;\;\; (d \geq 4)
\eea

\begin{thm} \label{thm:Ndelta}
Let $t$ and $q$ be related by
\be
t = \frac{q}{(1-q)^2}, \qquad q = \frac{1+2t-\sqrt{1+4t}}{2t}.
\ee
Then we have
\be \label{eqn:M&n}
\sum_{\delta \geq 0} n_\delta(x) t^\delta
= \frac{1}{(1-q)^{x^2-3x+2}(1-q^2)} \cdot
\sum_{j \geq 0} M_j(x) q^j
\ee
and
\be \label{eqn:NdeltaGen}
\sum_{\delta \geq 0} n_\delta(x) t^\delta
= \frac{1}{(1-q)^{x^2-3x} \cdot [\infty]!^3} \biggl(\binom{x+2}{2}
- 3 \sum_{i\geq 1} \frac{q^i}{(1-   q^i)^2} \biggr).
\ee
\end{thm}

\begin{proof}
We will use the following identity:
\be \label{eqn:SumIdentity}
\sum_{k \geq 0} \frac{\prod_{j=k}^{2k-1} (m-j)}{k!}t^k
= \frac{1}{\sqrt{1+4t}}\biggl(\frac{1+\sqrt{1+4t}}{2}\biggr)^{m+1}.
\ee
From this we get
\ben
\sum_{\delta \geq 0} n_\delta(x) t^\delta
& = & \sum_{\delta \geq 0} \sum_{j=0}^\delta
M_j(x) \cdot \binom{x^2-3x+3-\delta-j}{\delta-j} t^\delta \\
& = & \sum_{j \geq 0} M_j(x)q^j \sum_{k\geq 0} \binom{x^2-3x+3-2j - k}{k} t^k \\
& = & \sum_{j \geq 0} M_j(x)q^j \cdot \frac{1}{\sqrt{1+4t}}
\biggl(\frac{1+\sqrt{1+4t}}{2}\biggr)^{x^2-3x+4-2j} \\
& = &  \frac{1}{\sqrt{1+4t}}
\biggl(\frac{1+\sqrt{1+4t}}{2}\biggr)^{x^2-3x+4} \\
&& \cdot \sum_{j \geq 0} M_j(x)t^j
\biggl(\frac{1+\sqrt{1+4t}}{2}\biggr)^{-2j} \\
& = & \frac{1-q}{1+q} \cdot \frac{1}{(1-q)^{x^2-3x+4}} \cdot
\sum_{j \geq 0} M_j(x) q^j \\
& = &  \frac{1}{(1-q)^{x^2-3x+2}(1-q^2)} \cdot
\sum_{j \geq 0} M_j(x) q^j.
\een
This proves the first identity.
Now the second identity follows from the first one and Theorem \ref{thm:Mdelta} as follows:
\ben
&& \sum_{\delta \geq 0} n_\delta(x) t^\delta \\
& = & \frac{1-q}{1+q} \cdot \frac{1}{(1-q)^{x^2-3x+4}} \cdot
\frac{[1]^2[2]}{[\infty]!^3} \biggl(\binom{x+2}{2}
- 3 \sum_{i\geq 1} \frac{q^i}{(1-   q^i)^2} \biggr) \\
& = &  \frac{1}{(1-q)^{x^2-3x} \cdot [\infty]!^3} \biggl(\binom{x+2}{2}
- 3 \sum_{i\geq 1} \frac{q^i}{(1-   q^i)^2} \biggr).
\een
\end{proof}

Note there is an amazing similarity with G\"ottsche-Yau-Zaslow formula \cite{Got} (see also \eqref{eqn:GYZ}):
\ben
&& \sum_{\delta \geq 0}
T_\delta(L^2,L\cdot \kappa_S, c_1(S)^2, c_2(S)) \cdot (DG_2(q))^\delta \\
& = & \frac{(DG_2(q)/q)^{\chi(L)}B_1(q)^{K^2_S} B_2(q)^{L\cdot \kappa_S}}
{(\Delta(q)D^2G_2(q)/q^2)^{\chi(\cO_S)/2}},
\een
where $ D = q \frac{d}{d q}$,
$G_2$ is the second Eisenstein series
$$G_2(q) = -\frac{1}{24} + \sum_{n>0} \sum_{d|n} d \cdot q^n
= - \frac{1}{24} + \sum_{n=1}^\infty \frac{nq^n}{1-q^n}.$$
Note
\be
DG_2 =  \sum_{n=1}^\infty \frac{n^2q^n}{(1-q^n)^2},
\ee
and $t$ is the first term in the above expression of $DG_2$.

\subsection{Comparison with the predictions by   Katz-Klemm-Vafa}\label{sec:KKV}

In this subsection we will match our results with the prediction
of Katz-Klemm-Vafa \cite{KKV} in the case of local $\PP^2$,
where the formulas for $0 \leq \delta \leq 3$ were presented.
They have used \cite[(5.4)-(5.6)]{KKV},
We will instead use  \cite[(4.2),(4.15)]{KKV}.

Our notations are different from that in \cite{KKV}.
Denote by $\Ch^d(\PP^2)$ the Chow variety of degree $d$ plane algebraic curves.
By the genus formula,
a generic curve in this space has genus
$$g(d) = \frac{(d-1)(d-2)}{2}.$$
Each such curve is defined by a nontrivial homogeneous polynomial in three variables $x,y,z$
of degree $d$,
and the space of such polynomials have degree $\binom{d+2}{2}$,
so we know that $\Ch^d(\PP^2)$ is a projective space of dimension
\be
g(-d)-1= \frac{d^2+3d}{2}.
\ee
The prediction of \cite[(4.2)]{KKV} in the local $\PP^2$ case is
\be
n^{g(d)}_d(\kappa_{\PP^2}) = (-1)^{\dim \Ch^d(\PP^2)}e(\Ch^d(\PP^2))
= (-1)^{(d^2+3d)/2}\binom{d+2}{2}.
\ee

Denote by $\pi: \cC^d(\PP^2) \to \Ch^d(\PP^2)$ the universal curves,
and for each $j$, let $\pi^{(j)}: \cC^d(\PP^2)^{(j)} \to \Ch^d(\PP^2)$ be
the relative Hilbert schemes of $j$ points.
In other words,
each fiber of $\pi$ is a plane algebraic curve $C \subset \PP^2$,
and the corresponding fiber of $\pi^{(j)}$ is the Hilbert scheme of $j$ points on $C$.
It follows that there is a natural map
$$\varphi^{(j)}: \cC^d(\PP^2)^{(j)} \to (\PP^2)^{(j)},$$
because of a $0$-dimensional subscheme of length $j$ on a plane algebraic curve $C$
is automatically $0$-dimensional subscheme of length $j$ on $\PP^2$.
In the first paragraph of \cite[\S 8]{KKV},
the authors claimed that
 for $j \leq d+2$,
$\varphi^{(j)}$ is a fibration with fiber
$\PP^{d(d+3)/2-j}$,
then one has
\be
\begin{split}
e(\cC^d(\PP^2)^{(j)}) & = e(\PP^{d(d+3)/2-j}) \cdot e((\PP^2)^{(j)}) \\
& = \biggl( \binom{d+2}{2} - j \biggr) \cdot e((\PP^2)^{(j)}).
\end{split}
\ee
It follows that when $d \geq j-2$,
$e(\cC^d(\PP^2)^{(j)})$ is a quadratic polynomial $e_j(d)$ in $d$,
and we have the following generating series:
\be
\sum_{j \geq 0} e_j(x) q^j
= \sum_{j \geq 0} \biggl(\frac{(x+2)(x+1)}{2} - j\biggr) \cdot e((\PP^2)^{(j)}) q^j
\ee
By G\"{o}ttsche's formula,
\ben
\sum_{j \geq 0} e((\PP^2)^{(j)}) q^j
= \frac{1}{\prod_{n \geq 0} (1-q^n)^{e(\PP^2)}}
= \frac{1}{\prod_{n \geq 0} (1-q^n)^3}.
\een
By applying the operator $D = q \frac{d}{dq}$ we get:
\ben
\sum_{j \geq 0} j\cdot  e((\PP^2)^{(j)}) q^j
= - 3 \cdot \frac{1}{\prod_{n \geq 0} (1-q^n)^3} \sum_{n=1}^\infty \frac{nq^n}{1-q^n}.
\een
Hence
\be
\sum_{j \geq 0} e_j(x) q^j
= \frac{1}{[\infty]!^3}\biggl(\frac{(x+2)(x+1)}{2} - 3\sum_{n=1}^\infty \frac{nq^n}{1-q^n} \biggr).
\ee
By comparing with Theorem \ref{thm:Mdelta} we get

\begin{prop}
The polynomials $M_\delta(x)$ and $e_j(x)$ are related as follows:
\be \label{eqn:MinE}
\sum_{\delta \geq 0} M_\delta(x) q^\delta
= (1-q)^2(1-q^2) \cdot \sum_{j \geq 0} e_j(x) q^j.
\ee
\end{prop}

The prediction of \cite[(4.15)]{KKV} in the local $\PP^2$ case is
\be \label{eqn:KKVPred1}
n^{g(d)-\delta}_d(\kappa_{\PP^2}) = (-1)^{\dim \Ch^d(\PP^2)+\delta}
\sum_{j=0}^{\delta} b(g(d)-j, \delta-j) e(\cC^d(\PP^2)^{(j)}),
\ee
where
\be
b(g, k) = \begin{cases}
\frac{2}{k!} (g-1) \prod_{i=1}^{k-1} (2g-(k+2)-i), & k \geq 1, \\
1, & k = 0.
\end{cases}
\ee
Note there are some misprints in the formula for $b(g,k)$ in \cite{KKV},
but the correct expression can be found in \cite[(5.5)]{KKV}.
It is clear that
\be
b(g,k)
= \binom{2g-2-k}{k} + \binom{2g-3-k}{k-1}.
\ee
Since we have defined
the numbers
$n_\delta(d) = (-1)^{g(d)+d-1-\delta}n^{g(d)-\delta}_d(\kappa_{\PP^2})$,
we can rewrite \eqref{eqn:KKVPred1} as follows:
\be
n_\delta(d)= \sum_{j=0}^{\delta} b(g(d)-j, \delta-j) e(\cC^d(\PP^2)^{(j)}).
\ee
\begin{thm}
Let $t$ and $q$ be related by
\be
t = \frac{q}{(1-q)^2}, \qquad q = \frac{1+2t-\sqrt{1+4t}}{2t}.
\ee
Then we have
\be \label{eqn:NEC}
\sum_{\delta \geq 0} n_{\delta}(d) t^\delta \\
= \frac{1}{(1-q)^{d^2-3d}} \sum_{j \geq 0} q^j \cdot e(\cC^d(\PP^2)^{(j)}).
\ee
\end{thm}

\begin{proof}
We use \eqref{eqn:SumIdentity} to get
\ben
&& \sum_{\delta \geq 0} n_{\delta}(d) t^\delta \\
& = & \sum_{\delta \geq 0} t^\delta
\sum_{j=0}^{\delta}\biggl( \binom{2g(d)-2-j-\delta}{\delta-j} + \binom{2g(d)-3-j-\delta}{\delta-j-1} \biggr)
 \cdot e(\cC^d(\PP^2)^{(j)}) \\
& = & \sum_{j \geq 0} t^j\cdot e(\cC^d(\PP^2)^{(j)}) \\
&& \cdot \sum_{k \geq 0} \biggl( \binom{2g-2-2j-k}{k} + \binom{2g-3-2j-k}{k-1} \biggr) t^k \\
& = & \sum_{j \geq 0} t^j \cdot e(\cC^d(\PP^2)^{(j)}) \\
&& \cdot \frac{1}{\sqrt{1+4t}}\biggl( \biggl(\frac{1+\sqrt{1+4t}}{2} \biggr)^{2g(d)-1-2j}
+ t \cdot \biggl(\frac{1+\sqrt{1+4t}}{2} \biggr)^{2g(d)-3-2j} \biggr) \\
& = & \frac{1}{(1-q)^{d^2-3d}} \sum_{j \geq 0} q^j \cdot e(\cC^d(\PP^2)^{(j)}) .
\een
In the last equality we have used $t = \frac{q}{(1-q)^2}$.
\end{proof}

Now the prediction by Klemm-Katz-Vafa \cite{KKV}
for the local $\PP^2$ geometry can be formulated as follows:
For $\delta \leq d+2$,
\be \label{eqn:KKVPred}
\begin{split}
& (-1)^{(d^2+3d)/2+\delta} n^{g(d)-\delta}_d(\kappa_{\PP^2}) \\
= &
\sum_{j=0}^{\delta}\biggl( \binom{2g(d)-2-j-\delta}{\delta-j} + \binom{2g(d)-3-j-\delta}{\delta-j-1} \biggr)\\
& \cdot \biggl( \binom{d+2}{2} - j \biggr) \cdot e((\PP^2)^{(j)}).
\end{split}
\ee

\begin{thm}
The formula in Theorem \ref{thm:Ndelta} matches with \eqref{eqn:KKVPred}.
\end{thm}

\begin{proof}
By  \eqref{eqn:NdeltaGen},
\ben
\sum_{\delta \geq 0} n_{\delta}(d) t^\delta
& = & \frac{1}{(1-q)^{d^2-3d}} \sum_{j \geq 0} q^j \cdot  \biggl( \binom{d+2}{2} - j \biggr) \cdot e((\PP^2)^{(j)}) \\
& = & \frac{1}{(1-q)^{d^2-3d}} \frac{1}{[\infty]!^3}\biggl(\frac{(d+2)(d+1)}{2} - 3\sum_{n=1}^\infty \frac{nq^n}{1-q^n} \biggr),
\een
which is the right-hand side of \eqref{eqn:NdeltaGen}.
\end{proof}

\subsection{Another comparison with G\"ottsche-Yau-Zaslow formula for $\PP^2$}

Let us make a comparison with the work of Kool, Shende and Thomas \cite{KST}
on the number of $\delta$-nodal curves in a generic $\delta$-dimensional linear subsystem
$\PP^\delta_d \subset |\cO_{\PP^2}(d)|$.
Let $\cC^\delta_d \to \PP^\delta_d$ be the restriction of the universal family,
and let $(\cC^\delta_d)^{(j)} \to \PP^\delta_d$ be the relative Hilbert schemes.
Then by \cite[Theorem 3.4]{KST},
there is an expansion of the form:
\be
q^{1-g(d)} \sum_{j=0}^\infty e((\cC^\delta_d)^{(j)}) q^j = \sum_{r=g(d)-\delta}^{g(d)} n_r
t^{1-r},
\ee
where $t = \frac{q}{(1-q)^2}$ is exactly the variable $t$ we use in Theorem \ref{thm:Ndelta}.
Furthermore,
the leading coefficient $n_{g(d) - \delta}$ is the number of $\delta$-nodal curves in $\PP^\delta_d$.

\subsection{Three invariants associated to three bases}

We rewrite the identity \eqref{eqn:MinE} as follows.
First we have
\ben
\sum_{\delta \geq 0} M_\delta(x) q^{\delta+1}
=  \sum_{j \geq 0} e_j(x) \cdot (1-q)^2(1-q^2)q^{j+1},
\een
then change $q$ to $q^{-1}$ to get:
\ben
\sum_{\delta \geq 0} M_\delta(x) q^{-(\delta+1)}
=  \sum_{j \geq 0} e_j(x) \cdot (1-q^{-1})^2(1-q^{-2})q^{-j-1}.
\een
We then subtract the second identity from the first one and divide both sides by $q-q^{-1}$ to get
\ben
\sum_{\delta \geq 0} M_\delta(x) R_\delta
= - \sum_{j \geq 0} e_j(x) \cdot (q^{1/2}-q^{-1/2})^2(q^{j+3}+q^{-(j+3)}),
\een
This suggests to introduce a third basis of $SP(q)$ as follows:
\be
T_g: =\begin{cases}
1, & g= 0, \\
(q^{1/2}-q^{-1/2})^2, & g =1, \\
(q^{1/2}-q^{-1/2})^2 (q^{g-1}+q^{-(g-1)}), & g \geq 2.
\end{cases}
\ee

\begin{lem} \label{lm:SgTg}
The bases $\{S_g\}$ and $\{T_g\}$ are related as follows:
\bea
&& S_g = \sum_{j=1}^{g} (-1)^{g-j} \binom{2g-2}{g-j} T_j, \\
&& T_g = \sum_{j=1}^{g} \biggl(\binom{g+j-1}{g-j} - \binom{g+j-3}{g-2-j} \biggr) S_j.
\eea
\end{lem}

\begin{proof}
The first identity follows easily from the definitions by writing
$S_g = (q^{1/2}-q^{-1/2})^2\cdot (q^{1/2}-q^{-1/2})^{2g-2}$ and applying the binomial expansion
to $(q^{1/2}-q^{-1/2})^{2g-2}$.
For the second identity we apply \eqref{eqn:RgInSg} to get
\ben
T_g & = & (q^{1/2}-q^{-1/2})^2\cdot (R_{g-1} - R_{g-3}) \\
& = & (q^{1/2}-q^{-1/2})^2\cdot \biggl(\sum_{j=0}^{g-1} \binom{g+j}{g-1-j} S_j
- \sum_{j=0}^{g-3} \binom{g+j-2}{g-3-j} S_j \biggr) \\
& = & \sum_{j=0}^{g-1} \biggl(\binom{g+j}{g-1-j} - \binom{g+j-2}{g-3-j} \biggr) S_{j+1} \\
& = & \sum_{j=1}^{g} \biggl(\binom{g+j-1}{g-j} - \binom{g+j-3}{g-2-j} \biggr) S_j.
\een
\end{proof}

Define the numbers $\{E^h_d:\; 0 \leq h \leq g(d)\}$ as follows:
\be
\sum_{g=0}^{g(d)} (-1)^g n^g_d S_g
= \sum_{h=0}^{g(d)} E^h_d T_h.
\ee
As a corollary to the above Lemma,
we have

\begin{lem}
The numbers $\{n^g_d:\; 0 \leq g \leq g(d)\}$ and
the numbers $\{E^h_d:\; 0 \leq h \leq g(d)\}$ are related as follows:
\bea
&& n^g_d = (-1)^g \sum_{h = g}^{g(d)} E^h_d \cdot \biggl(
\binom{h+g-1}{h-g} - \binom{h+g-3}{h-g-2} \biggr), \\
&& E^h_d = (-1)^h \sum_{g=h}^{g(d)} \binom{2g-2}{g-h} \cdot n^g_d.
\eea
\end{lem}

In particular,
if we set $g=g(d)-\delta$ and $h=g(d) -j$,
we get
\be
\begin{split}
& n^{g(d)-\delta}_d = (-1)^{g(d)-\delta} \sum_{j=0}^{\delta}  E^{g(d)-j}_d \\
& \cdot \biggl(
\binom{2g(d)-\delta-j-1}{\delta-j} - \binom{2g(d)-\delta-j-3}{\delta-j-2} \biggr).
\end{split}
\ee
It is not hard to check that
\be
\begin{split}
& \binom{2g(d)-\delta-j-1}{\delta-j} - \binom{2g(d)-\delta-j-3}{\delta-j-2} \\
=& \binom{2g(d)-2-j-\delta}{\delta-j} + \binom{2g(d)-3-j-\delta}{\delta-j-1}.
\end{split}
\ee
So we have
\be
\begin{split}
& n^{g(d)-\delta}_d = (-1)^{g(d)-\delta} \sum_{j=0}^{\delta}  E^{g(d)-\delta}_d \\
=& \binom{2g(d)-2-j-\delta}{\delta-j} + \binom{2g(d)-3-j-\delta}{\delta-j-1}.
\end{split}
\ee
By comparing with \eqref{eqn:KKVPred},
we get the following geometric interpretation of $E^{g(d)-j}_d$:
\be
E^{g(d)-j}_d = (-1)^{d+1} e(\cC^d(\PP^2)^{(j)}).
\ee
To summarize,
we have discussed three bases of $SP(q)$:
\ben
&& R_g = q^g +q^{g-2} + \cdots + q^{-g}, \\
&& S_g = (q^{1/2}-q^{-1/2})^{2g}, \\
&& T_g = \begin{cases}
1, & g= 0, \\
\frac{1}{1+\delta_{g,1}} (q^{1/2}-q^{-1/2})^2 (q^{g-1}+q^{-(g-1)}), & g \geq 1,
\end{cases}
\een
and their associated invariants $N^g_d$, $n^g_d$ and $E^g_d$.
From the geometric point view, $E^g_d$ is most natural,
but from the representation theoretical point of view,
the introduction of $T_g$ is artificial and $R_g$ is most natural.

\section{Going Beyond the Predictions of Katz-Klemm-Vafa}

In this section we will refine our method in last section to compute
the Gopakumar-Vafa invariants $n^{(d-1)(d-2)/2-\delta}_d$
of $\kappa_{\P^2}$ for $0 \leq \delta \leq 2d-5$.
We will go beyond  the predictions made by Katz-Klemm-Vafa \cite{KKV}
for the case when $\varphi^{(j)}$'s are fibrations.
We have the following stronger version of Theorem \ref{thm:Mdelta}:
\begin{thm} \label{thm:Mdelta2}
For $\delta \geq 0$,
when $d \geq (\delta +5)/2$,
the transformed GV invariants $M^\delta_d(\kappa_{\PP^2}):=(-1)^{d-1}  N_d^{g(d)-\delta} (\kappa_{\PP^2})$
is almost a quadratic polynomial $M_\delta(d)$ in $d$,
and they have the following generating series
\be
\begin{split}
&\sum_{\delta\geq 0} M_\delta(d) q^\delta= \frac{[1]^2[2]}{[\infty]!^3} \biggl(\binom{d+2}{2}
- 3 \sum_{i\geq 1} \frac{q^{i}}{(1-   q^{i})^2}\biggr) \\
&- \frac{3\cdot q^{d-1}[2][3]}{[\infty]!^3}\biggl(\binom{d+1}{2}
- 3 \sum_{i\geq 1} \frac{q^{i}}{(1-   q^{i})^2}
-3\frac{q^{3}}{1-q^{3}} \biggr).
\end{split}
\ee
\end{thm}
This theorem is a corollary of Theorem \ref{F(d)2}.

Under the same conditions,
under the same condition,
the numbers $(-1)^{g(d)+d-1-\delta}n^{g(d)-\delta}_d(\kappa_{\PP^2})$ is almost a polynomial $n_\delta(d)$
given by:
\be
n_\delta(d): = \sum_{j=0}^\delta
M_j(d) \cdot \binom{d^2-3d+3-\delta-j}{\delta-j}.
\ee
As a consequence we then have
\begin{thm} \label{thm:Ndelta2}
Let $t$ and $q$ be related by
\be
t = \frac{q}{(1-q)^2}, \qquad q = \frac{1+2t-\sqrt{1+4t}}{2t}.
\ee
Then we have
\be \label{eqn:NdeltaGen2}
\begin{split}
& \sum_{\delta \geq 0} n_\delta(x) t^\delta
= \frac{1}{[1]^{x^2-3x} \cdot [\infty]!^3} \biggl(\binom{x+2}{2}
- 3 \sum_{i\geq 1} \frac{q^i}{(1-   q^i)^2} \biggr) \\
& - \frac{3\cdot q^{x-1}\cdot [3]}{[1]^{x^2-3x+2}[\infty]!^3}\biggl(\binom{x+1}{2}
- 3 \sum_{i\geq 1} \frac{q^{i}}{(1-   q^{i})^2}
-3\frac{q^{3}}{1-q^{3}} \biggr).
\end{split}
\ee
\end{thm}

\begin{proof}
Just recall \eqref{eqn:M&n}:
\ben
\sum_{\delta \geq 0} n_\delta(x) t^\delta
& = & \frac{1}{[1]^{x^2-3x+2}[2]} \cdot
\sum_{j \geq 0} M_j(x) q^j
\een
\end{proof}

Recall the prediction of \cite[(4.15)]{KKV} can be rephrased as follows
(cf. \eqref{eqn:NEC}):
\be
\sum_{\delta \geq 0} n_{\delta}(d) t^\delta \\
= \frac{1}{(1-q)^{d^2-3d}} \sum_{j \geq 0} q^j \cdot e(\cC^d(\PP^2)^{(j)}).
\ee
Hence by \eqref{eqn:NdeltaGen2} we get the following prediction for the Euler numbers of
relative Hilbert schemes:
\be
\begin{split}
&  \sum_{j \geq 0} q^j \cdot e(\cC^d(\PP^2)^{(j)})
= \frac{1}{[\infty]!^3} \biggl(\binom{d+2}{2}
- 3 \sum_{i\geq 1} \frac{q^i}{(1-   q^i)^2} \biggr) \\
& - \frac{3\cdot q^{d-1}\cdot [3]}{[1]^{2}[\infty]!^3}\biggl(\binom{d+1}{2}
- 3 \sum_{i\geq 1} \frac{q^{i}}{(1-   q^{i})^2}
-3\frac{q^{3}}{1-q^{3}} \biggr)
\end{split}
\ee
modulo $q^{2d-4}$.
It is interesting to establish this geometrically.

\subsection{A formula of $\I(d)$}
Recall that
by formula \eqref{Id}, we can write down the leading terms of $\I(d)$ in the following form:
\begin{align*}
\I(d)
 = & (-1)^d \cdot \sum_{\sum_i d_i=d}
q^{\frac{1}{2}\sum_i \kappa_{(d_i)}}
\W_{(d_1),(d_2)}(q)\W_{(d_2),(d_3)}(q)\W_{(d_3),(d_1)}(q) \nonumber\\
 & + (-1)^d \cdot 3 \sum_{\sum_i d_i=d}
q^{\frac{1}{2}( \kappa_{(d_1)}+\kappa_{(d_2-1,1)}+\kappa_{(d_3)})}\\
&\qquad\qquad \qquad \cdot\W_{(d_1),(d_2-1,1)}(q)\W_{(d_2-1,1),(d_3)}(q) \W_{(d_3),(d_1)}(q) \\
& +  \cdots\nonumber\\
= & (-1)^d q^{g(d)-1} \Big[ W_d(q) +  \I^{(2)}(q) + \cdots\Big],
\end{align*}
where
\begin{align*}
&\I^{(2)}(d)
=q^{-(g(d)-1)}\cdot 3 \sum_{\sum_i d_i=d}
q^{\frac{1}{2}( \kappa_{(d_1)}+\kappa_{(d_2-1,1)}+\kappa_{(d_3)})}\\
&\qquad\qquad \qquad \cdot\W_{(d_1),(d_2-1,1)}(q)\W_{(d_2-1,1),(d_3)}(q) \W_{(d_3),(d_1)}(q).
\end{align*}

\begin{prop}\label{degreeofId2} We have the following degree estimates:
\ben
\deg_q (\I(d)-(-1)^d q^{g(d)-1} \cdot W_d(q)  &=& (g(d)-1)- (d+2),
\\
\deg_q (\I(d)-(-1)^d q^{g(d)-1} [ W_d(q) +  \I^{(2)}(q) ]&=&
(g(d)-1)- (2d+1).
\een
\end{prop}
\begin{proof}
The first estimate is in fact a reformulation of Lemma \ref{degreeofId}.
For the second  one, note that for three partitions $\mu^1,\mu^2,\mu^3$ of $d_1, d_2, d_3$ respectively,
if they are not of the following forms:
\ben
&&(d_1),(d_2),(d_3)\\ &&(d_1-1,1),(d_2),(d_3) \quad \\ &&(d_1),(d_2-1,1),(d_3) \quad \\ && (d_1),(d_2),(d_3-1,1)
\een
then there are two cases: Case 1. They are of the forms
 \ben
&&(d_1-1,1),(d_2-1,1),(d_3) \quad \\ &&(d_1),(d_2-1,1),(d_3-1,1) \quad \\ && (d_1-1,1),(d_2),(d_3-1,1)
\een
Case 2. One of them say $\mu^i$ must have $\mu^i_1 \leq d_i-2$.

For the first case, recall
\ben
\W_{(d_1-1,1),(d_2-1,1)}(q)=q^{-\frac{1}{2}(d_1+d_2)}s_{(d_1-1,1)}(q^\varrho)
s_{(d_2-1,1)}(q^{(d_1-1,1)+\varrho}),
\een
where
$q^{\mu+\varrho}=(q^{\mu_1},q^{\mu_2-1},q^{\mu_3-2},\cdots)$.
By $s_{(d-1,1)} = h_{d-1}h_1-h_{d}$,
\begin{align*}
& s_{(d_2-1,1)}(q^{(d_1-1,1)+\varrho})\\
=&
 (q^{d_1-1}-q^{-1}+\frac{1}{1-q^{-1}}) ( q^{(d_1-1)(d_2-1)}+q^{(d_1-1)(d_2-2)}(1+q^{-2}+\cdots))\\&-
( q^{(d_1-1)(d_2)}+q^{(d_1-1)(d_2-1)}(1+q^{-2}+\cdots)).
\end{align*}
Hence
\begin{align*}
&\deg_q  \W_{(d_1-1,1),(d_2-1,1)}
= - \half(d_1+d_2)+(-1)+(d_2-1)(d_1-1),
\end{align*}
 we have
\begin{align*}
&\deg_q  \bigg[\sum_{\sum_i d_i=d}
q^{\frac{1}{2}( \kappa_{(d_1-1,1)}+\kappa_{(d_2-1,1)}+\kappa_{(d_3)})}\\
&\qquad \cdot\W_{(d_1-1,1),(d_2-1,1)}(q)\W_{(d_2-1,1),(d_3)}(q) \W_{(d_3),(d_1-1,1)}(q)\bigg]\\
=& (g(d)-1)- (d_1+d_2)- (d_1+d_2 +d_3+1+d_3+1) \\
=& (g(d)-1) -(2d+2).
\end{align*}
For the second case, without loss of generality suppose that $\mu^1_1\leq d_1-2$.  By Lemma \ref{lm:WEst},
$$
\deg_q \W_{\mu^1,\mu^2}(q) \leq d_1 \cdot d_2 - (d_1+d_2)/2-2(d_2+1),
$$
hence
\begin{align*}
&\deg_q  \bigg[ \sum_{\sum_i d_i=d}
q^{\frac{1}{2}( \kappa_{\mu^1}+\kappa_{\mu^2}+\kappa_{\mu^3})}
\W_{\mu^1,\mu^2}(q)\W_{\mu^2,\mu^3}(q) \W_{\mu^3,\mu^1}(q)\bigg]\\
\leq & (g(d)-1)- (2 d_1-1)-2(d_2+1+d_3+1) =(g(d)-1)- (2 d+1)
\end{align*}
The equality holds if $\mu_1 = (d_1-2,2)$ and $\mu_2=(d_2), \mu_3 = (d_3)$. This finishes the proof.
\end{proof}

Now we can write down the following formula for more leading terms
of $\I(d)$, which is in fact a generalization of Lemma \ref{Idwd} and Theorem \ref{thm:Wd}.
\begin{thm}\label{Id2} We have
\begin{align}
&\I(d)= \frac{1}{[\infty]_{q^{-1}}!^3}\biggl[\biggl(\binom{d+2}{2}
- 3 \sum_{i\geq 1} \frac{q^{-i}}{(1-   q^{-i})^2}\biggr)+\\
& \frac{3\cdot q^{-d-2}}{ (1-q^{-1})^2}\biggl(\binom{d+1}{2}
- 3 \sum_{i\geq 1} \frac{q^{-i}}{(1-   q^{-i})^2}
+3 \biggr)\biggr]
\biggl|_{q^{> -2d}}.
\nonumber
\end{align}
\end{thm}
We will finish the proof in the following two subsections.

\subsection{Contribution of $W_d$}

First we consider the $q^{> -2d}$ contributions from $W_d^2$ ,$W_d^3$ and $W_d^4$ . Since
$$
\deg_q W_d^4 \leq -(2d+3),
$$
it has no contribution.
Recall
\ben
 W_d^{2'}(q) &=&  W_d^2(q)-W_\infty^2(q) \\
&=& -3 \sum_{d_1+d_2>d} \frac{1}{([\infty]_{q^{-1}}!)^3}
\cdot T^x_{d_2}\left(\frac{(q^{-d_1-1};q^{-1})_{\infty}}{(xq^{-d_1-1};q^{-1})_{\infty}}-1\right), \\
 W_d^3(q)
&=& 3 \sum_{\sum d_i=d} \frac{1}{([\infty]_{q^{-1}}!)^3} \cdot T^x_{d_2}
\biggl(\frac{(q^{-d_1-1};q^{-1})_{\infty}}{(xq^{-d_1-1};q^{-1})_{\infty}}-1\biggr)\\
&& \qquad \qquad
\cdot T^x_{d_3}
\biggl(\frac{(q^{-d_2-1};q^{-1})_{\infty}}{(xq^{-d_2-1};q^{-1})_{\infty}}-1\biggr), \nonumber
\een
by Lemma \ref{lm:leading},
\begin{align*}
&\deg_q W_d^{2'}(q)=\max\{ -(d_1+1)(d_2+1) : d_1+d_2\geq d+1\}\\
&\deg_q W_d^{3}(q)=\max\{-(d_1+1)(d_2+1)-(d_2+1)(d_3+1): \sum_{i=1}^3 d_i=d\}
\end{align*}
If $d_1,d_2 \geq 1$, we have
\begin{align*}
&-(d_1+1)(d_2+1)=-(d_1-1)(d_2-1)-2(d_1+d_2)\leq - 2d.
\end{align*}
If $d_2 \geq 1$, we have
\begin{align*}
&-(d_1+1)(d_2+1)-(d_2+1)(d_3+1)\\
=& -(d+d_1d_2+d_2d_3+d_2+2)\leq - (2d+2) .
\end{align*}
Hence we only need to consider the terms of $d_1=0$ or $d_2=0$ (respectively $d_2=0$)
in the summation of $W_d^{2'}$ (respectively $W_d^3$).

Since
\begin{align*}
&T^x_{d_2}\left(\frac{(q^{-d_1-1};q^{-1})_{\infty}}{(xq^{-d_1-1};q^{-1})_{\infty}}-1\right)|_{d_2=0}\\
=&{(q^{-d_1-1};q^{-1})_{\infty}}-1
=-[\infty]_{q^{-1}}!(\frac{1}{[\infty]_{q^{-1}}!}-\frac{1}{[d_1]_{q^{-1}}!}), \\
&T^x_{d_3}
\biggl(\frac{(q^{-d_2-1};q^{-1})_{\infty}}{(xq^{-d_2-1};q^{-1})_{\infty}}-1\biggr)\biggr|_{d_2=0}
=T^x_{d_3}
\biggl(\frac{(q^{-1};q^{-1})_{\infty}}{(xq^{-1};q^{-1})_{\infty}}-1\biggr)\\
=&\sum_{k=0}^{d_3} \frac{q^{-k}\cdot[\infty]_{q^{-1}}!}{[k]_{q^{-1}}!}-1=-
[\infty]_{q^{-1}}!(\frac{1}{[\infty]_{q^{-1}}!}-\frac{1}{[d_3]_{q^{-1}}!}).
\end{align*}
We have the following formula of these terms:
\begin{align}
  W_d^{2'}(q)
=& \frac{2\cdot3}{([\infty]_{q^{-1}}!)^2}\sum_{d_1>d}
(\frac{1}{[\infty]_{q^{-1}}!}-\frac{1}{[d_1]_{q^{-1}}!})+q^{-2d}(\cdots),\label{wdb}\\
 W_d^3(q) =& \sum_{d_1+d_3=d} \frac{3}{[\infty]_{q^{-1}}!}
(\frac{1}{[\infty]_{q^{-1}}!}-\frac{1}{[d_3]_{q^{-1}}!})\label{wdc}
\\
& \qquad\quad\cdot(\frac{1}{[\infty]_{q^{-1}}!}-\frac{1}{[d_1]_{q^{-1}}!}) +q^{-2d-2}(\cdots).\nonumber
\end{align}
Moreover, we can take the summations to get the following:

\begin{prop}\label{wbwc}
We have
\begin{align}
  W_d^{2'}(q)
=& \frac{2\cdot 3 q^{-(d+2)}}{(1-q^{-1})^2([\infty]_{q^{-1}}!)^3}+  q^{-2d} \cdot(\cdots), \\
 W_d^3(q)
=& \frac{3 q^{-(d+2)}}{(1-q^{-1})^2([\infty]_{q^{-1}}!)^3}(d+1 -2\sum_{i\geq 2} \frac{q^{-i}}{1-q^{-i}}) \\
& +q^{-2d-2}(\cdots). \nonumber
\end{align}

\end{prop}
\begin{proof}
By formula \eqref{wdb} and
\begin{align*}
& \frac{-3}{([\infty]_{q^{-1}}!)^2}\sum_{d_1>d}
(\frac{1}{[\infty]_{q^{-1}}!}-\frac{1}{[d_1]_{q^{-1}}!})\\
=&  \frac{-3}{([\infty]_{q^{-1}}!)^2} \sum_{d_1>d}
(\sum_{k>d_1} \frac{q^{-k}}{[k]_{q^{-1}}!})\\
=& \frac{-3}{([\infty]_{q^{-1}}!)^2} \sum_{d_1>d}
(\sum_{k>d_1} \frac{q^{-k}}{[\infty]_{q^{-1}}!}-(\sum_{k>d_1} q^{-k}(\frac{1}{[\infty]_{q^{-1}}!}-\frac{1}{[k]_{q^{-1}}!}))\\
=& \frac{-3}{([\infty]_{q^{-1}}!)^3}
\frac{q^{-d-2}}{(1-q^{-1})^2} + \frac{3 q^{-2d-4}}{([\infty]_{q^{-1}}!)^2} \cdot(1+a_1q^{-1} +\cdots)
\end{align*}
we get the first formula. For the second one, by formula \eqref{wdc} we need to calculate:

\begin{align*}
& \sum_{d_1+d_3=d} \frac{3}{[\infty]_{q^{-1}}!}
(\frac{1}{[\infty]_{q^{-1}}!}-\frac{1}{[d_3]_{q^{-1}}!})
(\frac{1}{[\infty]_{q^{-1}}!}-\frac{1}{[d_1]_{q^{-1}}!})\\
=& \sum_{d_1+d_3=d} \frac{3}{[\infty]_{q^{-1}}!}
(\sum_{k>d_1} \frac{q^{-k}}{[k]_{q^{-1}}!})(\sum_{l>d_2} \frac{q^{-l}}{[l]_{q^{-1}}!})\\
=& \sum_{d_1+d_3=d} \frac{3}{[\infty]_{q^{-1}}!} \bigg[
(\sum_{k>d_1} \frac{q^{-k}}{[k]_{q^{-1}}!}-\frac{q^{-k}}{[\infty]_{q^{-1}}!})
 (\sum_{l>d_3} \frac{q^{-l}}{[l]_{q^{-1}}!}-\frac{q^{-l}}{[\infty]_{q^{-1}}!})+\\
 &2(\sum_{k>d_1} \frac{q^{-k}}{[k]_{q^{-l}}!}-\frac{q^{-k}}{[\infty]_{q^{-1}}!})
 (\sum_{l>d_3}\frac{q^{-l}}{[\infty]_{q^{-1}}!})+ (\sum_{k>d_1}\frac{q^{-l}}{[\infty]_{q^{-k}}!}) (\sum_{l>d_3}\frac{q^{-l}}{[\infty]_{q^{-1}}!}) \bigg]
 \\
=&  \sum_{d_1+d_3=d} q^{-(d_1+1+d_1+2+d_3+1+d_3+2)}(\cdots) +
 \frac{3 q^{-(d_1+1+d_3+1)}}{(1-q^{-1})^2([\infty]_{q^{-1}}!)^3}
\\
&\qquad +\sum_{0 \leq d_1 \leq d} 2(\sum_{k>d_1} \frac{q^{-k}}{[k]_{q^{-l}}!}-\frac{q^{-k}}{[\infty]_{q^{-1}}!})
 (\frac{q^{-(d-d_1)-1}}{(1-q^{-1})[\infty]_{q^{-1}}!})
 \\
=&  q^{-2d-6}(\cdots)+\frac{3 (d+1) q^{-(d+2)}}{(1-q^{-1})^2([\infty]_{q^{-1}}!)^3}+ \frac{6q^{-(d+2)}}{(1-q^{-1})([\infty]_{q^{-1}}!)^2}
\\
& \qquad\qquad \cdot
\sum_{0 \leq d_1 } q^{d_1+1}(\sum_{k>d_1} \frac{q^{-k}}{[k]_{q^{-l}}!}-\frac{q^{-k}}{[\infty]_{q^{-1}}!})
 +  q^{-2d-4}(\cdots).
\end{align*}
We have used Lemma \ref{lm:leading} repeatedly.
The proof is completed by using the following Lemma.
\end{proof}
\begin{lem}
The following identity holds:
$$
\sum_{d_1\geq 0} \sum_{k>d_1} q^{d_1+1-k}(\frac{1}{[\infty]_{q^{-1}}!}-\frac{1}{[k]_{q^{-1}}!})=\frac{1}{(1-q^{-1})[\infty]_{q^{-1}}!} \sum_{i\geq 2} \frac{q^{-i}}{1-q^{-i}}.
$$
\end{lem}
\begin{proof}
\begin{align*}
&\sum_{d_1\geq 0} q^{d_1+1} \sum_{k>d_1} q^{-k}(\frac{1}{[\infty]_{q^{-1}}!}-\frac{1}{[k]_{q^{-1}}!})=
\sum_{d_1\geq 0}q^{d_1+1}\sum_{k>d_1} (\sum_{l>k} \frac{q^{-k-l}}{[l]_{q^{-1}}!})\\=&
\sum_{d_1\geq 0}q^{d_1+1}\sum_{l>d_1+1} (\sum_{l>k>d_1} \frac{q^{-k-l}}{[l]_{q^{-1}}!})
=\sum_{d_1\geq 0}\sum_{l>d_1+1} \frac{q^{-l}}{[l]_{q^{-1}}!} \cdot \frac{1-q^{-(l-d_1-1)}}{1-q^{-1}}\\
=& \frac{1}{1-{q^{-1}}} \sum_{l > 1} \sum_{0\leq d_1 <l-1} \frac{q^{-l}}{[l]_{q^{-1}}!} - \frac{q^{-(2l-d_1-1)}}{[l]_{q^{-1}}!}  \\
=& \sum_{l > 1} \frac{1}{[l]_{q^{-1}}!} (\frac{(l-1) q^{-l}}{(1-q^{-1})}-
\frac{(q^{-l-1}-q^{-2l})}{(1-q^{-1})^2})\\
=& \sum_{l \geq 0} \frac{1}{[l]_{q^{-1}}!} \biggl(\frac{l q^{-l}}{1-q^{-1}}-
\frac{q^{-l}-q^{-2l}}{(1-q^{-1})^2}\biggr).
\end{align*}
Since
\ben
\sum_{d\geq 0}\frac{d q^{-d}}{[d]_{q^{-1}}!} &=& t\frac{d}{dt}\sum_{d\geq 0}\frac{t^d q^{-d}}{[d]_{q^{-1}}!}\biggr|_{t=1}
\\&=& t\frac{d}{dt} \frac{1}{(tq^{-1},q^{-1})_{\infty}} \biggr|_{t=1}
= \frac{1}{[\infty]_{q^{-1}}!} \sum_{i\geq 1} \frac{q^{-i}}{1-q^{-i}}.
\een
We have
\begin{align*}
 &\sum_{l \geq 0} \frac{1}{[l]_{q^{-1}}!} (\frac{l q^{-l}}{1-q^{-1}}-
\frac{q^{-l}-q^{-2l}}{(1-q^{-1})^2})\\
=& \frac{1}{(1-q^{-1})[\infty]_{q^{-1}}!} \sum_{i\geq 1} \frac{q^{-i}}{1-q^{-i}}
-
\frac{1}{(1-q^{-1})^2}\biggl(\frac{1}{[\infty]_{q^{-1}}!}-\frac{1-q^{-1}}{[\infty]_{q^{-1}}!} \biggr)\\
=& \frac{1}{(1-q^{-1})[\infty]_{q^{-1}}!} \sum_{i\geq 2}\frac{q^{-i}}{1-q^{-i}}.
\end{align*}
\end{proof}

\subsection{Contribution of $\I^{(2)}_d(q)$}

Introduce
\ben
&& \widetilde \W_{(m),(n)}(q)
= \sum^{n}_{k=0} \frac{q^{-k(m+1)}}{[m]_{q^{-1}}![k]_{q^{-1}}!}=\frac{1}{[\infty]_{q^{-1}}!}T^x_n\left(\frac{(q^{-m-1};q^{-1})_\infty}{(xq^{-m-1};q^{-1})_\infty}\right).
\een
By \eqref{Wmn},
\be
\widetilde \W_{(m),(n)}(q) = q^{-mn+(m+n)/2} \W_{(m),(n)}(q).
\ee
Because it has been proved in \cite{zhou2} that $\W_{\mu,\nu}(q)= \W_{\nu, \mu}(q)$,
it follows that $\widetilde \W_{(m),(n)}(q) = \widetilde \W_{(n),(m)}(q)$, i.e.
\be
\sum^{m}_{k=0} \frac{q^{-k(n+1)}}{[n]_{q^{-1}}![k]_{q^{-1}}!}
= \sum^{n}_{k=0} \frac{q^{-k(m+1)}}{[m]_{q^{-1}}![k]_{q^{-1}}!}.
\ee

\begin{lem}
The following identities hold:
\ben
&&\W_{(m),(n-1,1)}(q) \\
=&&  q^{m(n-1)-1-\frac{1}{2}(m+n)}\cdot \bigg[\frac{1}{1-q^{-1}}
\widetilde \W_{(m),(n-1)}(q) -
\frac{q^{-(m+1)(n-1)}}{[m]_{q^{-1}}![n]_{q^{-1}}!}\bigg].
\een
\end{lem}
\begin{proof}
Recall
\ben
\W_{(m),(n-1,1)}(q) & = & s_{(m)}(q^\varrho)
s_{(n-1,1)}(q^{(m)+\varrho}) \\
& = & \frac{q^{-m/2}}{[m]_{q^{-1}}!}\cdot
s_{(n-1,1)}(q^{(m)+\varrho}).
\een
Because $s_{(n-1,1)} = h_{n-1}h_1-h_{n}=s_{(n-1)} s_{(1)} - s_{(n)}$,
\ben
& & s_{(n-1,1)}(q^{(m)+\varrho})
= s_{(n-1)}(q^{(m)+\varrho})\cdot s_{(1)}(q^{(m)+\varrho}) - s_{(n)}(q^{(m)+\varrho})\\
&=&
q^{-(n-1)/2} \sum_{k=0}^{n-1} \frac{q^{m(n-1)-(m+1)k}}{[k]_{q^{-1}}!} \cdot q^{-1/2}(q^m+\frac{q^{-1}}{1-q^{-1}}) \\
& - & q^{-n/2} \sum_{k=0}^{n} \frac{q^{mn-(m+1)k}}{[k]_{q^{-1}}!}\\
&=& \frac{q^{mn-n/2-1}}{1-q^{-1}} \sum_{k=0}^{n-1} \frac{q^{-(m+1)k}}{[k]_{q^{-1}}!} -
\frac{q^{-3n/2}}{[n]_{q^{-1}}!}.
\een
Here we have used \eqref{eqn:sn(qm+rho)}.
\be
s_{(n)}(q^{(m)+\varrho})
= q^{-n/2} \sum_{k=0}^n \frac{q^{mn-(m+1)k}}{[k]_{q^{-1}}!}.
\ee
\end{proof}

By the above Lemma
we can write down the  contribution of $\I_d(q)$ as follow:
\begin{align*}
&  \I^{(2)}_d(q)=3 q^{-(g(d)-1)} \sum_{\sum_i d_i=d, d_2-1\geq 1} q^{\half (\kappa_{(d_1)}+\kappa_{(d_2-1,1)}+\kappa_{(d_3)})} \cdot \label{I2d}\\
& \quad\quad\quad\quad\quad\quad\quad\quad\quad\quad\quad\quad \W_{(d_1),(d_2-1,1)}\W_{(d_2-1,1),(d_3)}\W_{(d_3),(d_1)}\nonumber\\
=& 3 \sum_{\sum_i d_i=d, d_2\geq 2} q^{-(d+2)}
  \cdot (\frac{\widetilde\W_{(d_1),(d_2-1)}}{1-q^{-1}} -
\frac{q^{-(d_1+1)(d_2-1)}}{[d_1]_{q^{-1}}![d_2]_{q^{-1}}!})\nonumber\\
&\quad\quad\quad\quad\quad\quad\quad
\cdot (\frac{\widetilde\W_{(d_2-1),(d_3)}}{1-q^{-1}}-
\frac{q^{-(d_3+1)(d_2-1)}}{[d_3]_{q^{-1}}![d_2]_{q^{-1}}!})\cdot \W_{(d_3),(d_1)}\\
= & \frac{3q^{-d-2}}{(1-q^{-1})^2} \sum_{\sum_i d_i=d-1, d_2\geq 1} \widetilde\W_{(d_1),(d_2)}\widetilde\W_{(d_2),(d_3)}\widetilde\W_{(d_3),(d_1)}\\
& -3q^{-d-2}\sum_{\sum_i d_i=d, d_2\geq 2} \frac{2q^{-(d_3+1)(d_2-1)}\widetilde\W_{(d_1),(d_2-1)}\widetilde\W_{(d_3),(d_1)}}{(1-q^{-1})[d_1]_{q^{-1}}![d_2]_{q^{-1}}!}
+ q^{-d-1}(\cdots) .
\end{align*}
We can rewrite it as
\begin{align*}
& \I^{(2)}_d(q)=\I^{(2a)}_d(q)
+\I^{(2a')}_d(q)+\I^{(2b)}_d(q)+ q^{-d-1}(\cdots) .
\end{align*}
where

\begin{align*}
\I^{(2a)}_d(q)&= \frac{3q^{-d-2}}{(1-q^{-1})^2} \sum_{\sum_i d_i=d-1} \widetilde\W_{(d_1),(d_2)}\widetilde\W_{(d_2),(d_3)}\widetilde\W_{(d_3),(d_1)}\\
 \I^{(2a')}_d(q)&=  -\frac{3q^{-d-2}}{(1-q^{-1})^2}  \sum_{d_1+d_3=d-1} \widetilde\W_{(d_1),(0)}\widetilde\W_{(0),(d_3)}\widetilde\W_{(d_3),(d_1)}\\
\I^{(2b)}_d(q)&= -6q^{ -(d+2)}\sum_{\sum_i d_i=d, d_2\geq 2} \frac{q^{-(d_3+1)(d_2-1)}\widetilde\W_{(d_1),(d_2-1)}\widetilde\W_{(d_3),(d_1)}}
{(1-q^{-1})[d_3]_{q^{-1}}![d_2]_{q^{-1}}!}\\
\end{align*}

\begin{prop}\label{I2b}
We have
\begin{align*}
\I^{(2b)}_d(q)=&\frac{-6 q^{ -(d+2)} \cdot q^{-1} }{(1-q^{-1})^3 ([\infty]_{q^{-1}}!)^3} + q^{-2d}(\cdots)
\end{align*}
\end{prop}
\begin{proof}

\begin{align*}
&\sum_{\sum_i d_i=d, d_2\geq 2} \frac{q^{-(d_3+1)(d_2-1)}\widetilde\W_{(d_1),(d_2-1)}\widetilde\W_{(d_3),(d_1)}}
{(1-q^{-1})[d_3]_{q^{-1}}![d_2]_{q^{-1}}!}
\\
=&\sum_{\sum_i d_i=d, d_2\geq 2} \frac{q^{-(d_3+1)(d_2-1)}}{(1-q^{-1})[d_3]_{q^{-1}}![d_2]_{q^{-1}}!}
\\ &\qquad \cdot \frac{1}{[\infty]_{q^{-1}}!}T^x_{d_2}\left(\frac{(q^{-d_1-1};q^{-1})_\infty}{(xq^{-d_1-1};q^{-1})_\infty}\right).
\frac{1}{[\infty]_{q^{-1}}!}T^x_{d_3}\left(\frac{(q^{-d_2-1};q^{-1})_\infty}{(xq^{-d_2-1};q^{-1})_\infty}\right),
\end{align*}
and
$$
T^x_{d_2}\left(\frac{(q^{-d_1-1};q^{-1})_\infty}{(xq^{-d_1-1};q^{-1})_\infty}\right)
-1 = q^{(d_2+1)(d_1+1)}(\cdots),
$$
we have
$$
\I^{(2b)}_d(q) = -6q^{-d-2} \sum_{d_2+d_3\leq d, d_2\geq 2} \frac{q^{-(d_3+1)(d_2-1)}}{(1-q^{-1})[d_3]_{q^{-1}}![d_2]_{q^{-1}}!
([\infty]_{q^{-1}}!)^2}+ q^{-2d-2}(\cdots) .
$$
We now take the summations.
By the q-analog of binomial theorem:
\begin{align*}
&\sum_{d_3\geq 0} \frac{q^{-(d_3+1)(d_2-1)}}{[d_3]_{q^{-1}}!
} = q^{-(d_2-1)}\frac{1}{(q^{-(d_2-1)};q^{-1})_{\infty}}
=q^{-(d_2-1)}\frac{[d_2-2]_{q^{-1}}!}{[\infty]_{q^{-1}}!},
\end{align*}
we have
\begin{align*}
&\sum_{d_2\geq 2, d_3\geq 0} \frac{q^{-(d_3+1)(d_2-1)}}{(1-q^{-1})[d_3]_{q^{-1}}![d_2]_{q^{-1}}!
} \\
=&
\frac{1}{(1-q^{-1}) ([\infty]_{q^{-1}}!)}\sum_{d_2\geq 2} \frac{q^{-(d_2-1)}}{(1-q^{-(d_2-1)})(1-q^{-d_2})
}\\
=& \frac{1}{(1-q^{-1})^2 ([\infty]_{q^{-1}}!)}\sum_{d_2\geq 2} ( \frac{1}{1-q^{-(d_2-1)}}
-\frac{1}{1-q^{-d_2}})\\
=& \frac{1}{(1-q^{-1})^2 ([\infty]_{q^{-1}}!)} (\frac{1}{1-q^{-1}} -1)
= \frac{q^{-1}}{(1-q^{-1})^3 ([\infty]_{q^{-1}}!)}.
\end{align*}
Noting the fact that
$$
 \sum_{d_2+d_3>d, d_2\geq 2} \frac{q^{-(d_3+1)(d_2-1)}}{(1-q^{-1})[d_3]_{q^{-1}}![d_2]_{q^{-1}}!
([\infty]_{q^{-1}}!)^2
}= q^{-d+1}(\cdots),
$$
we finish the proof.
\end{proof}
\begin{prop}\label{I2aa} We have
\ben
\I^{(2a')}_d(q)= -\frac{3q^{-(d+2)}}{(1-q^{-1})^2 ([\infty]_{q^{-1}}!)^3}
\biggl(d-2 \sum_{i\geq 1}\frac{q^{-i}}{1-q^{-i}}\biggr) \biggl|_{q^{\geq -2d-1}}.
\een
\end{prop}
\begin{proof}
\begin{align*}
&\sum_{d_1+d_3=d-1} \widetilde\W_{(d_1),(0)}\widetilde\W_{(0),(d_3)}\widetilde\W_{(d_3),(d_1)}\\
=&\sum_{d_1+d_3=d-1}
\sum^{d_3}_{k=0} \frac{q^{-k(d_1+1)}}{([d_1]_{q^{-1}}!)^2[d_3]_{q^{-1}}![k]_{q^{-1}}!}\\
=&\sum_{d_1+d_3=d-1} \frac{1}{[d_1]_{q^{-1}}![d_3]_{q^{-1}}!{[\infty]_{q^{-1}}!}}
(T^x_n\left(\frac{(q^{-m-1};q^{-1})_\infty}{(xq^{-m-1};q^{-1})_\infty}\right)
-1)
\\
&\qquad +\sum_{d_1+d_3=d-1} \frac{1}{[d_1]_{q^{-1}}![d_3]_{q^{-1}}![\infty]_{q^{-1}}!}.
\end{align*}
By Lemma \ref{lm:leading} and the following lemma, we finish the proof.
\end{proof}
\begin{lem} We have
$$\sum_{d_1+d_3=d} \frac{1}{[d_1]_{q^{-1}}![d_3]_{q^{-1}}!}
=
\frac{1}{([\infty]_{q^{-1}}!)^2} \biggl((d+1)-2 \sum_{i\geq 1}\frac{q^{-i}}{1-q^{-i}}\biggr) \biggl|_{q^{\geq -d}}
$$
\end{lem}
\begin{proof}
\ben
\sum_{d_1+d_3=d} \frac{1}{[d_1]_{q^{-1}}![d_3]_{q^{-1}}!}
=
\sum_{d_1+d_3=d} \bigg[ -\frac{2}{[\infty]_{q^{-1}}!}
(\frac{1}{[\infty]_{q^{-1}}!}-\frac{1}{[d_3]_{q^{-1}}!})\\
\qquad\qquad +\frac{1}{([\infty]_{q^{-1}}!)^2} +
(\frac{1}{[\infty]_{q^{-1}}!}-\frac{1}{[d_1]_{q^{-1}}!})
(\frac{1}{[\infty]_{q^{-1}}!}-\frac{1}{[d_3]_{q^{-1}}!})
\bigg]
\een
By Proposition \ref{E1} and Lemma \ref{lm:leading},
$$
\sum_{d\geq d_1\geq 0}  \left(\frac{1}{[\infty]_{q^{-1}}!}-\frac{1}{[d_3]_{q^{-1}}!} \right)
= \frac{1}{[\infty]_{q^{-1}}!} \sum_{i\geq 1} \frac{q^{-i}}{1-q^{-i}}+ q^{-(d+1)}(\cdots),
$$
hence
\ben
\sum_{d_1+d_3=d} \frac{1}{[d_1]_{q^{-1}}![d_3]_{q^{-1}}!}
&=&
\frac{1}{([\infty]_{q^{-1}}!)^2} \biggl((d+1)-2 \sum_{i\geq 1}\frac{q^{-i}}{1-q^{-i}}\biggr)\\
&&+ q^{-(d+1)}(\cdots)+ q^{-(d_1+1+d_2+1)}(\cdots).
\een
\end{proof}

\noindent
\textit{Proof of Theorem \ref{Id2} .}
By Proposition \ref{degreeofId2},
we need to consider:
\begin{align*}
\I(d) = & (-1)^d q^{g(d)-1} \Big[ [W_d^1(q) +W_\infty^2(q)] +  \I^{(2a)}(q)\\
& +[W_d^{2'}(q) +W_d^3(q)  +  \I^{(2a')}(q) +  \I^{(2b)}(q) ]+ q^{-2d}(\cdots)\Big].
\end{align*}
Recall we have shown that (cf. \eqref{eqn:W1} and \eqref{eqn:W2inf}):
$$
 W_d^1(q) +W_\infty^2(q) =
 \frac{1}{[\infty]_{q^{-1}}!^3}  \biggl(\binom{d+2}{2}
- 3 \sum_{i\geq 1} \frac{q^{-i}}{(1- q^{-i})^2}\biggr).
$$
Also by Theorem \ref{thm:Wd} we have
\ben
 &&\I^{(2a)}(q) = \frac{3q^{-d-2}}{(1-q^{-1})^2} \cdot W_{d-1}(q)\\
 & =&
  \frac{1}{[\infty]_{q^{-1}}!^3} \frac{3q^{-d-2}}{(1-q^{-1})^2} \biggl(\binom{d+1}{2}
- 3 \sum_{i\geq 1} \frac{q^{-i}}{(1- q^{-i})^2}\biggr)\biggl|_{q^{\geq -2d-2}}
\een
By Proposition \ref{wbwc}, Proposition \ref{I2aa} and Proposition \ref{I2b} proved in the last two subsections, we have
\begin{align*}
&W_d^{2'}(q) +W_d^3(q)  +  \I^{(2a')}(q) +  \I^{(2b)}(q)
\\=&
\frac{2\cdot 3 q^{-(d+2)}}{(1-q^{-1})^2([\infty]_{q^{-1}}!)^3}+ \frac{3 q^{-(d+2)}}{(1-q^{-1})^2([\infty]_{q^{-1}}!)^3}(d+1 -2\sum_{i\geq 2} \frac{q^{-i}}{1-q^{-i}})\\
&-\frac{3q^{-(d+2)}}{(1-q^{-1})^2 ([\infty]_{q^{-1}}!)^3} (d-2 \sum_{i\geq 1}\left(\frac{q^{-i}}{1-q^{-i}}\right))+\frac{-6 q^{ -(d+2) \cdot q^{-1}} }{(1-q^{-1})^3 ([\infty]_{q^{-1}}!)^3}\\
&+ q^{-2d}(\cdots)\\
=& \frac{3 q^{-(d+2)}}{(1-q^{-1})^2([\infty]_{q^{-1}}!)^3}(2+1 )+\frac{-6 q^{ -(d+2)} (q^{-1}-q^{-1}) }{(1-q^{-1})^3 ([\infty]_{q^{-1}}!)^3}+ q^{-2d}(\cdots)\\
=& \frac{9 q^{ -(d+2)} }{(1-q^{-1})^2 ([\infty]_{q^{-1}}!)^3}+ q^{-2d}(\cdots)
\end{align*}
We finish the proof by the above three terms.
\qed
\subsection{Contribution of $\I(1)\I(d-1)$}
\begin{prop}\label{I1Id}
 We have
\ben
&&\I(1)\I(d-1) \\
&=& \frac{3q^{g(d)-1-(d-1)}}{(1-q^{-1})^2 [\infty]_{q^{-1}}!^3}  \biggl(\binom{d+2}{2}
- 3 \sum_{i\geq 1} \frac{q^{-i}}{(1-   q^{-i})^2}\biggr)\biggl|_{q^{\geq -2d+1}}.
\een
\end{prop}
\begin{proof}
By definition,
$$
\widetilde \W_{(1),(0)}(q)
= \frac{1}{(1-q^{-1})}, \quad \widetilde \W_{(0),(0)}(q)
= 1,
$$
we have
$$
W_1(q)=\sum_{d_1+d_2+d_3=1 } \widetilde \W_{(d_1),(d_2)}\widetilde \W_{(d_2),(d_3)}\widetilde \W_{(d_3),(d_1)}= \frac{3}{(1-q^{-1})^2},
$$
and
$$
\I(1) = q^{-1} W_1(q) = \frac{3q^{-1}}{(1-q^{-1})^2}.
$$
Hence
\ben
\I(1)\I(d-1) = \frac{3q^{g(d-1)-1-1}}{(1-q^{-1})^2} W_{d-1}(q) =
\frac{3q^{g(d)-1-(d-1)}}{(1-q^{-1})^2} W_{d-1}(q),
\een
together with Theorem \ref{thm:Wd} we finish the proof.
\end{proof}

\begin{thm}\label{F(d)2} If $d>3$, we have
\begin{align*}
&F^{1}(d)=F(d)= \frac{1}{[\infty]_{q^{-1}}!^3}\biggl[\biggl(\binom{d+2}{2}
- 3 \sum_{i\geq 1} \frac{q^{-i}}{(1-   q^{-i})^2}\biggr)-\\
& \frac{3\cdot q^{-d+1}(1-q^{-3})}{ (1-q^{-1})^2}\biggl(\binom{d+1}{2}
- \sum_{i\geq 1} \frac{3q^{-i}}{(1-   q^{-i})^2}
- \frac{3q^{-3}}{1-q^{-3}} \biggr)\biggr]
\biggl|_{q^{> -(2d-4)}}\nonumber
\end{align*}
\end{thm}
\begin{proof} Recall
\ben
&& \deg_q(\I(d_1) \cdots \I(d_k))
= \half [(d_1^2-3d_1) + \cdots + (d_k^2-3d_k)] \\
& = &  \half(d^2-3d) - \sum_{1 \leq i < j \leq k} d_id_j
.
\een
If $\{d_1,\cdots,d_k\}\neq \{1,d-1\}$ or $\{d\}$, there must be some $d_s\leq d-2$,
$$
\sum_{1 \leq i < j \leq k} d_id_j \geq d_s\cdot (\sum_{i \neq s}d_i) = d_s(d-d_s)\geq 2d-4 .
$$
Hence we have
\ben
F(d) = \I(d) - \I(1)\I(d-1) + q^{-(2d-4)} (\cdots),
\een
together with formula \eqref{F(d)=N2}, Theorem \ref{Id2} and Proposition \ref{I1Id} we finish the proof.
\end{proof}

\subsection{Further observations on quadratic properties}
Our results in this and last sections have suggested
the following form of the generating series for the transformed GV invariants:
\be
\begin{split}
\sum_{\delta\geq 0} M^\delta_d q^\delta& = \sum_{ \delta \geq 0} C_{d,0}^\delta q^\delta
+ q^{d-1} \cdot \sum_{ \delta \geq 0} C_{d,1}^\delta q^\delta  \\
& + q^{2d-4} \cdot \sum_{ \delta \geq 0} C_{d,2}^\delta q^\delta
+ q^{3d-9} \cdot \sum_{ \delta \geq 0} C_{d,3}^\delta  q^\delta
+ \cdots,
\end{split}
\ee
where  for sufficiently large $d$,
\be
C_{d,j}^\delta
= a_j(\delta) \binom{d+2-j}{2}
+ b_j(\delta).
\ee
I.e. we have
$$
 M^\delta_d = C_{d,0}^{\delta} + C_{d,1}^{\delta-(d-1)} + C_{d,2}^{\delta-(2d-4)} + C_{d,3}^{\delta-(3d-9)} +\cdots  .
 $$
In the above we have established this for $C^\delta_{d,0}$ and $C^\delta_{d,1}$.
Using the table of $n^g_d$ for local $\PP^2$ in \cite{HKR}
we have checked that
\ben
C^0_{d,2} & = & 6 \cdot \binom{d}{2}, \qquad(d \geq 5), \\
C^1_{d,2} & = & 12\cdot \binom{d}{2} - 18, \qquad (d \geq 6), \\
C^2_{d,2} &= & 24\cdot \binom{d}{2} - 90, \qquad (d \geq 6), \\
C^3_{d,2} & = & 30 \cdot \binom{d}{2} -252, \qquad(d \geq 7), \\
C^4_{d,2} & = & 33\cdot \binom{d}{2} - 549, \qquad (d \geq 7), \\
C^5_{d,2} &= & -15\cdot \binom{d}{2} - 882, \qquad (d \geq 8).
\een
It will be interesting to have geometric interpretation
of such phenomenon.

\end{document}